\documentclass{amsart}

\usepackage{amssymb}
\usepackage[mathscr]{eucal}
\usepackage{graphicx}
\usepackage{mathrsfs}
\usepackage{psfrag}
\usepackage{fullpage}
\usepackage{color, harvard}

\theoremstyle{plain}
\newtheorem{theorem}{Theorem}[section]

\newtheorem{lemma}[theorem]{Lemma}

\newtheorem{condition}[theorem]{Condition}

\newtheorem{remark}[theorem]{Remark}

\newcommand{\lb}{\left\{}
\newcommand{\rb}{\right\}}
\newcommand{\Def}{\overset{\text{def}}{=}}

\newcommand{\R}{\mathbb{R}}
\newcommand{\N}{\mathbb{N}}

\newcommand{\KK}{K}

\newcommand{\Borel}{\mathscr{B}}
\newcommand{\Pspace}{\mathscr{P}}
\newcommand{\BP}{\mathbb{P}}
\newcommand{\BE}{\mathbb{E}}
\newcommand{\filt}{\mathscr{F}}

\newcommand{\la}{\left \langle}
\newcommand{\ra}{\right\rangle}

\newcommand{\ee}{\mathfrak{e}}
\newcommand{\genA}{\mathcal{A}}

\newcommand{\genL}{\mathcal{L}}
\newcommand{\jump}{\mathcal{J}}

\newcommand{\pp}{\mathsf{p}}

\newcommand{\PP}{\mathcal{P}}

\newcommand{\NN}{{N,n}}
\newcommand{\mart}{\mathcal{M}}
\newcommand{\dfi}{\textsc{m}}
\newcommand{\QQ}{\mathcal{Q}}

\allowdisplaybreaks
\begin{document}

\title{Fluctuation Analysis for the Loss From Default}

\author{Konstantinos Spiliopoulos}
\address{Department of Mathematics \& Statistics\\
Boston University\\
Boston, MA 02215}
\email{kspiliop@math.bu.edu}

\author{Justin A. Sirignano}
\address{Department of Management Science and Engineering\\
Stanford University\\
Stanford, CA 94305}
\email{jasirign@stanford.edu}

\author{Kay Giesecke}
\address{Department of Management Science and Engineering\\
Stanford University\\
Stanford, CA 94305}
\email{giesecke@stanford.edu}

\date{\today. We are grateful to Andrew Abrahams, Jose Blanchet, Terry Lyons, and Marek Musiela for insightful comments. We also thank participants of seminars at Columbia University, University of Michigan, Rutgers University, the Oxford-Man Institute, the University of Southern California, the Office of Financial Research, the SIAM Conference on Financial Mathematics and Engineering, and the INFORMS Annual Meeting for comments.}

\begin{abstract}
We analyze the fluctuation of the loss from default around its large portfolio limit in a class of reduced-form models of correlated firm-by-firm default timing. We prove a weak convergence result for the fluctuation process and use it for developing a conditionally Gaussian approximation to the loss distribution. Numerical results illustrate the accuracy and computational efficiency of the approximation.
\end{abstract}

\maketitle

\section{Introduction}
Reduced-form point process models of correlated firm-by-firm default timing are widely used to measure the credit risk in portfolios of defaultable assets such as loans and corporate bonds. In these models, defaults arrive at intensities governed by a given system of stochastic differential equations.  Computing the distribution of the loss from default in these models is often challenging. Transform methods are tractable only in special cases, for example, if defaults are assumed to be conditionally independent. Monte Carlo simulation methods have a much wider scope but can be slow for the large portfolios and longer time horizons common in industry practice. A major US bank might easily have $20,000$ wholesale loans, $50,000$ to $100,000$ mid-market and commercial loans, and derivatives trades with $10,000$ to $20,000$ different legal entities. Simulation of such large pools is extremely burdensome.

This paper develops a conditionally Gaussian approximation to the distribution of the loss from default in large pools. 
The approximation applies to a broad class of empirically motivated reduced-form models in which a name defaults at an intensity that is influenced by an idiosyncratic risk factor process, a systematic risk factor process $X$ common to all names in the pool, and the portfolio loss. It is based on an analysis of the fluctuation of the loss around its large portfolio limit, i.e., a central limit theorem (CLT). More precisely, we show that the signed-measure-valued process describing the fluctuation of the loss around its law of large numbers (LLN) limit weakly converges to a unique distribution-valued process in a certain weighted Sobolev space. The limiting fluctuation process satisfies a stochastic evolution equation driven by the Brownian motion governing the systematic risk $X$ and a distribution-valued martingale that is centered Gaussian {\it given} $X$. The fluctuation limit, and thus the resulting approximation to the loss distribution, is Gaussian only in the special case that the names in the pool are not sensitive to $X$.
In the general case, the approximation is conditionally Gaussian. 

The weak convergence result proven in this paper extends the LLNs for the portfolio loss established in our earlier work \citeasnoun{GieseckeSpiliopoulosSowers2011} and \citeasnoun{GieseckeSpiliopoulosSowersSirigano2012}. The fluctuation analysis performed here involves challenging topological issues that do not arise in the analysis of the LLN. Firstly, the fluctuations process takes values in the space of signed measures. This space is not well suited for the study of convergence in distribution; the weak topology, which is the natural topology to consider here, is in general not metrizable (see, for example, \citeasnoun{BarrioDeheuvelsGeer2007}). We address this issue by analyzing the convergence of the fluctuations as a process taking values in the dual of a suitable Sobolev space of test functions. Weights are introduced in order to  control the norm of the fluctuations and establish tightness. Related ideas appear in \citeasnoun{Metivier1985}, \citeasnoun{FernandezMeleard}, \citeasnoun{KurtzXiong} and other articles, but for other systems and under different assumptions. Additional issues are the growth and degeneracies of the coefficients of our stochastic system, which make it difficult to identify the appropriate weights for the Sobolev space in which the convergence is established and uniqueness of the limiting fluctuation process is proved. \citeasnoun{Purtukhia1984}, \citeasnoun{GyongiKrylov1992}, \citeasnoun{KryLot1998}, \citeasnoun{Lot2001}, \citeasnoun{Kim2009}, and others face similar issues in settings that are different from ours. The approach we pursue is inspired by the class of weights that were introduced by \citeasnoun{Purtukhia1984} and further analyzed by \citeasnoun{GyongiKrylov1992}.

The fluctuation analysis leads to a second-order approximation to the loss distribution that can be significantly more accurate than an approximation obtained from just the LLN. Our numerical results, which are based on a method of moments for solving the stochastic evolution equation governing the fluctuation limit, confirm that the second-order approximation is much more accurate.  The second-order approximation is even accurate for relatively small portfolios in the order of hundreds of names or when the influence of the systematic risk process X is relatively low. The second-order approximation also improves the accuracy in the tails. The effort of computing the second-order approximation exceeds that of computing the first-order approximation, but is still much lower than directly simulating the pool.

Prior research has established various weak convergence results for interacting particle systems represented by reduced-form models of correlated default timing. \citeasnoun{daipra-etal} and \citeasnoun{daipra-tolotti} study mean-field models in which an intensity is a function of the portfolio loss. In a model with local interaction, \citeasnoun{giesecke-weber} take the intensity of a name as a function of the state of the names in a specified neighborhood of that name. In these formulations, the impact of a default on the dynamics of the surviving constituent names, a contagion effect, is permanent. In our mean-field model, an intensity depends on the path of the portfolio loss. Therefore, the impact of a default may be transient and fade away with time. There is a recovery effect. Moreover, the system which we analyze includes firm-specific sources of default risk and addresses an additional source of default clustering, namely the exposure of a name to a systematic risk factor process. This exposure generates a random limiting behavior for the LLN and leads to a fluctuation limit which is only conditionally Gaussian.

There are several other related articles. \citeasnoun{CMZ} prove a LLN for a mean-field system with permanent default impact, taking an intensity as a function of an idiosyncratic risk factor, a systematic risk factor, and the portfolio loss rate. The risk factors follow diffusion processes whose coefficients may depend on the portfolio loss. \citeasnoun{hambly} establish a LLN for a system in which defaults occur at first hitting times of correlated diffusion processes. Conditional on a correlating systematic risk factor governed by a Brownian motion, defaults occur independently of one another. 
\citeasnoun{PapanicolaouSystemicRisk} analyze large deviations in a system of diffusion processes that interact through their empirical mean and have a stabilizing force acting on each of them. \citeasnoun{fouque-ichiba} study defaults and stability in an interacting diffusion model of inter-bank lending.

The rest of this paper is organized as follows. Section \ref{S:Model} describes the class of reduced-form models of correlated default timing which we analyze.
Section \ref{S:LLN} reviews the law of large numbers proved by \citeasnoun{GieseckeSpiliopoulosSowersSirigano2012} for the portfolio loss in these models. Section \ref{S:Fluctuations} states our main weak convergence result for the fluctuation process, essentially a central limit theorem for the loss. Section \ref{S:NumericalMethod} provides a numerical method for solving the stochastic evolution equation governing the fluctuation limit and provides numerical results. Sections \ref{S:ProofMainCLT}-\ref{S:Uniqueness} are devoted to the proof of the main result. There is an appendix.

\section{Model and Assumptions}\label{S:Model}

We analyze a class of reduced-form point process models of correlated default timing in a pool of firms (``names''). Models of this type have been studied by \citeasnoun{CMZ}, \citeasnoun{daipra-etal}, \citeasnoun{daipra-tolotti}, \citeasnoun{GieseckeSpiliopoulosSowers2011}, \citeasnoun{GieseckeSpiliopoulosSowersSirigano2012}, and others. In these models, a default is governed by the first jump of a point process. The point process intensities follow a system of SDEs.

Fix a probability space  $(\Omega,\filt,\BP)$. Let $\{W^n\}_{n\in \N}$ be a countable collection of independent standard Brownian motions.  Let $\{\ee_n\}_{n\in \N}$ be an i.i.d. collection of standard exponential random variables which are independent of the $W^n$'s.  Finally, let $V$ be a standard Brownian motion which is independent of the $W^n$'s and $\ee_n$'s.  Each $W^n$ will represent a source of risk which is idiosyncratic to a specific name.  Each $\ee_n$ will represent a normalized default time for a specific name.  The process $V$ will drive a systematic risk factor process to which all names are exposed. Define $\mathcal{V}_{t}=\sigma\left(V_{s}, 0\leq s\leq t\right)\vee \mathcal{N}$ and $\filt^{n}_t=\sigma\left( (V_{s},W_{s}^{n}), 0\leq s\leq t\right) \vee \mathcal{N}$, where $\mathcal{N}$ contains the $\BP$-null sets. Lastly, we will denote by $\BP_{Y}$ the conditional law given $Y$, where $Y$ may represent $\mathcal{V}_{t}$, for example. 

Fix $N\in\N$, $n\in\{1,2,\ldots, N\}$ and consider the following system:
\begin{equation} \label{E:main}
\begin{aligned}
d\lambda^\NN_t &= -\alpha_\NN (\lambda^\NN_t-\bar \lambda_\NN)dt + \sigma_\NN \sqrt{\lambda^\NN_t}dW^n_t +  \beta^C_\NN dL^N_t+ \beta^S_\NN \lambda^\NN_t dX_t, \qquad \lambda^\NN_0 = \lambda_{\circ, N,n}\\
dX_t &= b_{0}( X_t) dt + \sigma_{0}(X_t)dV_t,\qquad X_0= x_\circ\\ 
\tau^\NN &= \inf\lb t\ge 0: \int_{0}^t \lambda^\NN_s ds\ge \ee_n\rb\\
L^N_t &= \frac{1}{N}\sum_{n=1}^N \chi_{\{\tau^\NN\le t\}}.
\end{aligned}
\end{equation}
Here, $\chi$ is the indicator function. The initial value $x_\circ$ of $X$ is fixed. The $\alpha_\NN,\bar \lambda_\NN,\sigma_\NN,\beta^C_\NN,\beta^S_\NN, \lambda_{\circ, N,n}$ are parameters.\footnote{\citeasnoun{giesecke-schwenkler} develop and analyze likelihood estimators of the parameters of point process models such as (\ref{E:main}), given a realization of $L^N$ over some sample period.}
The process $L^N$ is the fraction of names in default, which we loosely call the ``loss rate'' or simply the ``loss.''\footnote{The term ``loss rate'' can be taken literally if a name corresponds to a unit-notional position in the portfolio and the recovery at default is $0$.} The process $\lambda^\NN$ represents the intensity, or conditional default rate, of the $n$-th name in the pool.  More precisely, $\lambda^\NN$ is the density of the Doob-Meyer compensator to the default indicator $1-\dfi^\NN_t$, where $\dfi^\NN_t = \chi_{\{\tau^\NN>t\}}$. That is, a martingale is given by
$\dfi^\NN_t + \int_{0}^t \lambda^\NN_s \dfi^\NN_s ds.$
The results in Section 3 of \citeasnoun{GieseckeSpiliopoulosSowers2011} imply that the system (\ref{E:main}) has a unique  solution such that $\lambda^\NN_t\ge 0$ for every $N\in\N$, $n\in\{1,2,\ldots, N\}$ and $t\ge 0$. Thus, the model is well-posed.

The default timing model (\ref{E:main}) addresses several channels of default clustering. An intensity is influenced by an idiosyncratic source of risk represented by a Brownian motion $W^n$, and a source of systematic risk common to all firms--the diffusion process $X$.  Movements in $X$ cause correlated changes in firms' intensities and thus provide a channel for default clustering emphasized by \citeasnoun{ddk} for corporate defaults in the U.S.  The sensitivity of $\lambda^\NN$ to changes in $X$ is measured by the parameter $\beta^S_\NN\in\R$.  The second channel for default clustering is modeled through the feedback (``contagion'') term $\beta^C_\NN dL^N_t$.  A default causes a jump of size $\beta^C_\NN/N$ in the intensity $\lambda^\NN$, where $\beta^C_\NN\in \R_+= [0,\infty)$.  Due to the mean-reversion of $\lambda^\NN$, the impact of a default fades away with time, exponentially with rate $\alpha_\NN\in\R_+$. \citeasnoun{azizpour-giesecke-schwenkler} have found self-exciting effects of this type to be an important channel for the clustering of defaults in the U.S., over and above any clustering caused by the exposure of firms to systematic risk factors.

Figure \ref{fig:paths} illustrates the behavior of the system when the systematic risk $X$ follows an OU process. It shows sample paths of the processes $\lambda^\NN \dfi^\NN$ and $L^N$ for a pool with $N=4$ names. Between defaults, the intensities $\lambda^\NN$ of the surviving names evolve as correlated diffusion processes, where the co-movement is driven by the systematic risk $X$. At a default, the process $\lambda^\NN \dfi^\NN$ associated with the defaulting name drops to 0. At the same time, the processes $\lambda^\NN \dfi^\NN$ associated with the surviving names jump by $1/2$, and the loss $L^N$ increases by $1/4$.


We allow for a heterogeneous pool; the intensity dynamics of each name can be different.  We capture these different dynamics by defining the parameter ``types''
\begin{equation*}\label{E:typedef} \pp^\NN \Def (\alpha_\NN,\bar \lambda_\NN,\sigma_\NN,\beta^C_\NN,\beta^S_\NN). \end{equation*}
The $\pp^\NN$'s take values in parameter space $\PP\Def \R_+^4\times \R$. For each $N\in \N$, define
\begin{equation*}
\hat \pp^\NN_t \Def (\pp^\NN,\lambda^\NN_t)
\end{equation*}
for all $n\in \{1,2,\dots, N\}$ and $t\ge 0$. Define $\hat \PP\Def \PP\times \R_+$. The vector $\hat{\underline{\pp}}_{0}^{N}=(\hat \pp^{N,1}_0,\ldots,\hat \pp^{N,N}_0)$ represents a random environment that addresses the heterogeneity of the system.

\begin{condition}\label{A:regularity}
We assume that the $\hat\pp^{\NN}_{0}$ are i.i.d. random variables with
common law $\nu$ and that $\nu$ has compact support in $\hat{\PP}$. In particular, and to avoid inessential complications, we assume that all elements of the random vector $\pp^{\NN}_0$ are bounded in absolute values by some  $K$, for all $n\in \{1,2,\dots, N\}$. Moreover, we assume that $\{\hat\pp^{\NN}_{0}\}$ is independent of $\{W^{n}\}, \{\ee_n\}$ and $V$.
\end{condition}

This formulation of heterogeneity generalizes that of \citeasnoun{GieseckeSpiliopoulosSowersSirigano2012}. They assume that $\pi=\lim_{N\to \infty} \frac{1}{N}\sum_{n=1}^N \delta_{\pp^\NN}$ and $\Lambda_\circ=\lim_{N\to \infty} \frac{1}{N}\sum_{n=1}^N \delta_{\lambda_{\circ,N,n}}$ exist \textup{(}in $\Pspace(\PP)$ and $\Pspace(\R_+)$, respectively\textup{)}. Their formulation is a special case of the one proposed here, where the law $\nu=\pi\times\Lambda_\circ$.

\begin{condition}\label{A:integrability}
We assume that $\BE\int_{0}^{t}\left[|b_{0}(X_{s})|^{2}+|\sigma_{0}(X_{s})|^{4}\right]ds<\infty$ for all $t\ge 0$.
\end{condition}

\section{Law of large numbers}\label{S:LLN}
Except for special cases of little practical interest, the distribution of the portfolio loss $L^{N}$ in the system (\ref{E:main}) is difficult to compute. We are interested in an approximation to this distribution for the large portfolios common in practice, i.e., for the case that $N$ is large. \citeasnoun{GieseckeSpiliopoulosSowersSirigano2012} prove a law of large numbers (LLN) for the loss in the system (\ref{E:main}) and use it for developing a first-order approximation. We first review this result and then extend it in Section \ref{S:Fluctuations} by analyzing the fluctuations of the loss around its LLN limit. The fluctuation analysis will allow us to construct a more accurate second-order approximation.

To outline the LLN, define
\begin{equation*}
\mu^N_t \Def \frac{1}{N}\sum_{n=1}^N\delta_{\hat \pp^\NN_t}\dfi^\NN_t;
\end{equation*}
this is the empirical distribution of the type and intensity for those names which are still ``alive.''
We note that $\mu^N_t$ is a sub-probability measure. Since
\begin{align}\label{ln}
L^N_t =1-\mu^N_t(\hat \PP),
\end{align}
it suffices to study the limiting behavior of the measure-valued process $\{\mu^{N}_{t},t\in[0,T]\}_{N\in\N}$ for some fixed horizon $T>0$.
Let $E$ be the collection of sub-probability measures (i.e., defective probability measures) on $\hat \PP$; i.e., $E$ consists
of those Borel measures $\nu$ on $\hat \PP$ such that $\nu(\hat \PP)\le 1$.
Topologizing $E$ in the usual way (by projecting onto the one-point compactification of $\hat \PP$; see Chapter 9.5 of \citeasnoun{MR90g:00004}) we obtain that $E$ is a Polish space.
Thus, $\mu^N$ is an element of $D_E[0,\infty)$ where $D$ is the Skorokhod space (i.e., $D_E[0, \infty)$ is the set of RCLL processes on $[0, \infty)$ taking values in $E$).

Further, for $\hat \pp=(\pp,\lambda)$ where
$\pp=(\alpha,\bar \lambda,\sigma,\beta^C,\beta^S)\in \PP$ and $f\in
C^{2}_{b}(\hat \PP)$ (the space of twice continuously differentiable, bounded functions), define, similarly to \cite{GieseckeSpiliopoulosSowers2011}, the operators
\footnote{At this point we would like to remark that there is a typo in the
formulation of the corresponding operators in \cite{GieseckeSpiliopoulosSowers2011}.
In particular, it is mentioned there that $(\genL_2 f)( \pp) =  \frac{\partial f}{\partial \lambda}( \pp)$ and $\QQ( \pp) =  \beta^C\lambda$,
where it should have been $(\genL_2 f)( \pp) = \beta^C \frac{\partial f}{\partial \lambda}( \pp)$ and $\QQ( \pp) = \lambda$.}
\begin{align*}\label{E:Operators1}
(\genL_1 f)(\hat \pp) &= \frac12 \sigma^{2}\lambda\frac{\partial^2 f}{\partial \lambda^2}(\hat \pp) - \alpha(\lambda-\bar \lambda)\frac{\partial f}{\partial \lambda}(\hat \pp)-\lambda f(\hat \pp)\\
(\genL_2 f)(\hat \pp) &= \beta^C\frac{\partial f}{\partial \lambda}(\hat \pp)\\
(\genL_3^{x} f)(\hat \pp) &= \beta^{S}\lambda b_{0}(x)\frac{\partial f}{\partial \lambda}(\hat \pp)+\frac{1}{2}(\beta^{S})^{2}\lambda^{2}\sigma_{0}^{2}(x)\frac{\partial^{2}f}{\partial \lambda^{2}}(\hat \pp)\\
(\genL_4^{x} f)(\hat \pp) &=
\beta^{S}\lambda\sigma_{0}(x)\frac{\partial f}{\partial
\lambda}(\hat \pp).
\end{align*}
Also define
\begin{equation*} \QQ(\hat \pp) \Def \lambda.\end{equation*}
The generator $\genL_1$ corresponds to the diffusive part of the
intensity with killing rate $\lambda$, and $\genL_2$ is the
macroscopic effect of contagion on the surviving intensities at any
given time.  The operators $\genL_3^{x}$ and $\genL_4^{x}$ are related
to the exogenous systematic risk $X$.
For a measure $\nu\in D_E[0,\infty)$, we also specify the inner product
\begin{equation*}  \big{<} f, \nu \big{>}_{E} = \int_{\hat \PP} f(\hat \pp)  d\nu(\hat \pp). \end{equation*}

The law of large numbers of \citeasnoun{GieseckeSpiliopoulosSowersSirigano2012} states that $\mu^{N}_{t}$ weakly converges to $\bar{\mu}_{t}$ in $D_E[0,T]$. To rigorously formulate the result we need to use the weak form (see also Lemma \ref{L:Qchar}). In particular, for all $f\in C^{2}_{b}(\hat \PP)$, the evolution of
$\bar{\mu}_{\cdot}$ is governed by the measure evolution equation
\begin{equation}
d\la f,\bar \mu_t\ra_E =  \left\{\la \genL_1f,\bar \mu_t\ra_E+ \la \QQ,\bar \mu_t\ra_E
\la \genL_2f,\bar \mu_t\ra_E+\la \genL^{X_{t}}_3 f,\bar
\mu_t\ra_E\right\}dt+\la \genL^{X_{t}}_4 f,\bar \mu_t\ra_E dV_{t},\quad
 \text{ a.s.}\label{Eq:LimitMu}
\end{equation}
The LLN suggests an approximation to the distribution of the loss $L^N$ in large pools by the large pool limit:
\begin{align}\label{firstorder}
L_t^N \overset{d} \approx L_t &= 1- \bar{\mu}_t(\hat \PP).
\end{align}

\section{Main Result: Fluctuations Theorem}\label{S:Fluctuations}
In order to improve the first-order approximation (\ref{firstorder}), we analyze the fluctuations of $\mu^N$ around its large pool limit $\bar{\mu}$.
As is indicated by the proof of Lemmas \ref{L:KeyEstimate} and \ref{L:KeyEstimate2} (see also \citeasnoun{GieseckeSpiliopoulosSowersSirigano2012}), for an appropriate metric, the sequence $\{\sqrt{N}(\mu^{N}_{t}-\bar{\mu}_{t}): N<\infty\}$ is stochastically bounded.
Hence, it is reasonable to define the scaled fluctuation process $\Xi^{N}$ by
\begin{equation}\label{fluct-process}
\Xi^{N}_{t}= \sqrt{N}(\mu^{N}_{t}-\bar{\mu}_{t}).
\end{equation}
In Theorem \ref{T:MainCLT} below we will show that the signed-measure-valued process $\Xi^{N}$ weakly converges to a fluctuation limit process $\bar{\Xi}$ in an appropriate space.

The analysis of the limiting behavior of the fluctuation process (\ref{fluct-process}) involves issues that do not occur in the treatment of the LLN.
In particular, even though the fluctuation process is a signed-measure-valued process, its limit process $\{\bar{\Xi}_{\cdot}\}$ is distribution-valued in an appropriate space.
The space of signed measures endowed with the weak topology is not metrizable.\footnote{In general, a topological space is Polish if and only if its topology can be defined by
 a metric for which it is complete and separable. A locally compact space is metrizable if and only if it has a countable base in which it is Polish. The weak convergence of
 probability measures (i.e., finite non-negative measures) on a separable metric space can be defined by the Prokhorov metric. Thus, the weak topology on the space of probability
measures on a locally compact space $E$ is metrizable if and only if $E$ is Polish. However, issues arise if one replaces the  space of finite non-negative measures with the space
 of finite signed measures. The space of bounded and continuous functions on $E$, endowed with the sup-norm is a Banach space. Its topological dual endowed with the weak$^{*}$-topology
 coincides with the set of finite signed measures on $E$ endowed with the weak topology. However, the topological dual of an infinite dimensional Banach space is not metrizable,
 even though, by the Banach-Alaoglu theorem, any weak$^{*}$-compact subset of such a topological dual will be metrizable. For a more thorough discussion of these issues, see
\citeasnoun{BarrioDeheuvelsGeer2007}, Remark 1.2.}   The difficulty is then to identify a rich enough space, where tightness and uniqueness can be proven. It turns out that
we have to consider the convergence in weighted Sobolev spaces. Here, several technical challenges arise. These are mainly due to
 the growth and degeneracies of the coefficients of the system (\ref{E:main}), which make it difficult to identify the correct weights. The spaces that we consider are Hilbert spaces. For
the sake of clarity of presentation the appropriate Hilbert spaces will be defined in detail in Section \ref{S:SobolevSpace}. For the moment, we mention that the space in question is denoted by
 $W^{J}_{0}(w,\rho)$, with $w$ and $\rho$ the appropriate weight functions, $J\in\N$ and  $W^{-J}_{0}(w,\rho)$ will be its dual. A precise definition is given in Section \ref{S:SobolevSpace}.

We need to introduce several additional operators to state our weak convergence result. For $\hat \pp=(\pp,\lambda)\in \hat \PP\subset \mathbb{R}^{6}$ and $f\in C^{2}_{b}(\hat \PP)$, we define
\begin{align*}\label{E:Operators2}
(\mathcal{G}_{x,\mu}f)(\hat \pp)&=(\genL_1f)(\hat \pp)+ (\genL^{x}_3 f)(\hat \pp)+ \la \QQ, \mu\ra_{E} (\genL_2f)(\hat \pp)+\la \genL_2f, \mu\ra_{E}\QQ(\hat \pp)\\
(\genL_5 (f,g))(\hat \pp) &=\sigma^{2} \frac{\partial f}{\partial \lambda}(\hat \pp)\frac{\partial g}{\partial \lambda}(\hat \pp)\lambda\\
(\genL_6 (f,g))(\hat \pp) &=f(\hat \pp) g(\hat \pp)\lambda\\
(\genL_7 f)(\hat \pp) &=f(\hat \pp)\lambda\\
\end{align*}

The main result of this paper is the following theorem.
\begin{theorem}\label{T:MainCLT}
Let $D=\frac{\text{dim}(\hat{\PP})}{2}=3$. For $J>0$ large enough (in particular for $J>3D+1$) and for weight functions $(w,\rho)$ such that Condition \ref{A:AssumptionsOnWeights} holds, the sequence $\{\Xi^{N}_{t}, t\in[0,T]\}_{N\in \N}$ is relatively compact in $D_{W_{0}^{-J}(w,\rho)}[0,T]$. For any subsequence of this sequence,
 there exists a subsubsequence that converges in distribution with limit $\{\bar{\Xi}_{t},t\in[0,T]\}$. Any accumulation point $\bar{\Xi}$ satisfies the  stochastic evolution equation
\begin{equation}\label{Eq:CLT}
\la f,\bar \Xi_t\ra = \la f,\bar \Xi_0\ra+ \int_{0}^{t}\la \mathcal{G}_{X_{s},\bar{\mu}_{s}}f,\bar
\Xi_s\ra ds+\int_{0}^{t}\la \genL^{X_{s}}_4 f,\bar \Xi_s\ra dV_{s}+\la f,\bar{\mart}_t\ra, \text{ a.s.}
\end{equation}
for any $f\in W_{0}^{J}(w,\rho)$, where $\bar{\mart}$ is a distribution-valued martingale with predictable variation process
\begin{equation*}
[\la f,\bar{\mart}\ra]_t=\int_{0}^{t}\left[\la \genL_5 (f,f),\bar \mu_s\ra+ \la \genL_6 (f,f),\bar \mu_s\ra+\la \genL_2 f,\bar \mu_s\ra^{2}\la \QQ,\bar \mu_s\ra-2\la \genL_7 f,\bar \mu_s\ra\la \genL_2 f,\bar \mu_s\ra\right]ds.
\end{equation*}
Moreover, conditional on the $\sigma$-algebra  $\mathcal{V}_{t}$, $\bar{\mart}_t$ is centered Gaussian with covariance function, for $f,g\in W_{0}^{J}(w,\rho)$,
  given by
\begin{align}
\mathrm{Cov}\left[\la f, \bar{\mart}_{t_1}\ra, \la g, \bar{\mart}_{t_2}\ra\,\Big|\, \mathcal{V}_{t_1 \vee t_2}\right]&=\BE\bigg[\int_{0}^{t_1 \wedge t_2}\left[\la \genL_5 (f,g),\bar \mu_s\ra+ \la \genL_6 (f,g),\bar \mu_s\ra+\la \genL_2 f,\bar \mu_s\ra\la \genL_2 g,\bar \mu_s\ra\la \QQ,\bar \mu_s\ra\right.\nonumber\\
& \hspace{1cm}\left.-\la \genL_7 g,\bar \mu_s\ra\la \genL_2 f,\bar \mu_s\ra-\la \genL_7 f,\bar \mu_s\ra\la \genL_2 g,\bar \mu_s\ra\right]ds\, \Big|\,\mathcal{V}_{t_1 \vee t_2}\bigg].\label{Eq:ConditionalCovariation}
\end{align}
Finally, the limiting stochastic evolution equation (\ref{Eq:CLT}) has a unique solution in $W_{0}^{-J}(w,\rho)$ and thus the limit accumulation point $\bar{\Xi}_{\cdot}$ is unique.
\end{theorem}

\begin{remark}
Clearly, when $\beta^{S}_{\NN}=0$ for all $\NN$, then the limiting distribution-valued martingale $\bar{\mart}$ is centered Gaussian with covariance operator given by the (now deterministic) term within the expectation in (\ref{Eq:ConditionalCovariation}). Also, we remark that the operators $\mathcal{G}_{x,\mu}$ and $\genL_4^{x}$ in the stochastic evolution equation (\ref{Eq:CLT}) are linear. Conditionally on the systematic risk $X$, equation (\ref{Eq:CLT}) is linear.
\end{remark}

The proof of Theorem \ref{T:MainCLT} is developed in Sections \ref{S:ProofMainCLT} through \ref{S:Uniqueness}. In Section \ref{S:ProofMainCLT}, we identify the limiting equation and prove the convergence theorem based on the tightness and uniqueness results of Sections \ref{S:Tightness} and Section \ref{S:Uniqueness}. In Section
 \ref{S:SobolevSpace}, we discuss the Sobolev spaces we are using.  In Section \ref{S:Tightness},
we prove that the family $\{\Xi^{\NN}_{t}, t\in[0,T]\}_{N\in \N}$ is relatively compact in  $D_{W^{-J}_{0}(w,\rho)}[0,T]$, see Lemma \ref{L:XiPrelimitSpace}. In Section \ref{S:Uniqueness}, we prove uniqueness of (\ref{Eq:CLT}) in $W_{0}^{-J}(w,\rho)$, see Theorem \ref{T:SEE_Uniqueness}.

The fluctuation analysis leads to a second-order approximation to the distribution of the portfolio loss $L^{N}$ in large pools. The weak convergence established in Theorem \ref{T:MainCLT}  implies that
\begin{equation*}\label{ApproxMain00}
\mathbb{P}(\sqrt{N}(L^{N}_{t}-L_{t})\geq \ell)\approx \mathbb{P}(\bar{\Xi}_{t}(\hat \PP)\leq-\ell)
\end{equation*}
for large $N$. Theorem \ref{T:MainCLT} and its proof in Section \ref{S:ProofMainCLT} imply that $(\Xi^N,V,\bar{\mu})$ weakly converges to $(\bar{\Xi},V,\bar{\mu})$. This motivates the approximation
\begin{align*}
\mu^N_t = \frac{1}{\sqrt{N}}\Xi^N_{t} + \bar{\mu}_t \overset{d} \approx \frac{1}{\sqrt{N}} \bar{\Xi}_{t} + \bar{\mu}_t,
\end{align*}
which implies a second-order approximation for the portfolio loss:
\begin{align}\label{ApproxMain}
L_t^N \overset{d} \approx L_t - \frac{1}{\sqrt{N}} \bar{\Xi}_{t}(\hat \PP).
\end{align}
The next section develops, implements and tests a numerical method for computing the distribution of $L_t - \bar{\Xi}_{t}(\hat \PP)/\sqrt{N}$, and numerically demonstrates the accuracy of the approximation (\ref{ApproxMain}).

\section{Numerical Solution}\label{S:NumericalMethod}
The numerical solution of the stochastic evolution equation for the fluctuation limit is not standard.  There surprisingly exists little literature on solving this class of problems.  Solutions do not exist in $L^2$, so standard Galerkin methods are not applicable.\footnote{An example of the unique challenges posed by these equations is the simple case of i.i.d. Brownian motions starting from zero on the real line.  The fluctuation limit $\zeta_t(x)$ for this system satisfies the stochastic evolution equation $d  \big{<} f, \zeta_t \big{>} = \frac{1}{2} \big{<} f'', \zeta_t \big{>} dt + d \big{<} f, \mathcal{W}_t \big{>}$, where $\mathcal{W}_t$ is a Gaussian process with covariance $\textrm{Cov} [ \big{<} f, \mathcal{W}_t \big{>}, \big{<} g, \mathcal{W}_s \big{>} ] = \int_0^{s \wedge t} \int_{\mathbb{R}} f'(x) g'(x) (2 \pi u)^{-1/2} e^{- x^2/(2 u)}   dx du $.  Suppose there exists a solution $\zeta(t,x)$ which satisfies the SPDE $d  \zeta= \frac{1}{2} \frac{\partial^2 \zeta}{\partial x^2} dt + d \mathcal{W}_t$. Then, $\zeta_t(x) = \zeta(t,x) dx$ and one could solve the stochastic evolution equation by solving this SPDE.  However, challenges are immediately evident.  The Green's function solution $\zeta(t,x) = \int_0^t \int_{\mathbb{R}} dt' d x' G(t',x'; t,x) d \mathcal{W}_{t'}$ to the SPDE has infinite variance, indicating that the Green's solution is not in fact a solution to this SPDE.  It is interesting to note that the lack of a Green's function solution stems directly from the covariance structure; if one had $f(x) g(x)$ instead of $f'(x) g'(x)$, there would be a finite variance Green's function solution.}

\subsection{Method of Moments}
We provide a method of moments for solving for the fluctuation limit $\bar \Xi(\hat \PP)$. Along with the LLN limit $\bar\mu(\hat \PP)=1-L$, the fluctuation limit yields the approximation (\ref{ApproxMain}). The method extends the numerical approach for computing $\bar\mu(\hat \PP)$ developed by \citeasnoun{GieseckeSpiliopoulosSowersSirigano2012}. We first present it for a homogeneous portfolio; i.e., set $\nu=\delta_{\hat \pp_{0}}$ where $\hat \pp_0 \in \hat \PP$. Write $\bar\Xi_t(d \lambda)$ for the solution of the stochastic evolution equation (\ref{Eq:CLT}); for notational convenience, we do not explicitly show the dependence on the fixed parameters. For $k \in \mathbb{N}$, we define the ``fluctuation moments''
\begin{equation*}
v_{k}(t)=\int_{0}^{\infty} \lambda^{k} \bar \Xi_t(d \lambda).
\end{equation*}
We are interested in $v_{0}(t)$, which is equal to the fluctuation limit $\bar \Xi_t(\R_+)$. We also define the moments $u_{k}(t)$ of the solution $\bar\mu_t( d\lambda)$ of the stochastic evolution equation (\ref{Eq:LimitMu}), again not explicitly showing the dependence on the fixed parameters:
\begin{equation*}
u_{k}(t)=\int_{0}^{\infty}\lambda^{k} \bar \mu_t( d\lambda).
\end{equation*}
The zero-th moment gives the limiting loss $L$ since $L_t = 1 - u_0(t)$. Both the zero-th LLN moment $u_0(t)$ and the zero-th fluctuation moment $v_0(t)$ are required to compute the approximation (\ref{ApproxMain}).  The moments $\{u_{k}(t) \}_{k=0}^{\infty}$ satisfy a system of SDEs,
\begin{equation}
\begin{aligned}
d u_k(t)  &= \big\{ u_k(t) \big{(} - \alpha k + \beta^S b_0(X_t) k + 0.5 (\beta^S)^2 \sigma_0^2 (X_t) k (k-1)  \big{)}   \\
&+ u_{k-1}(t) \big{(} 0.5 \sigma^2 k(k-1) + \alpha \bar{\lambda} k + \beta^C k u_1(t) \big{)} - u_{k+1}(t) \big\} dt  + \beta^S \sigma_0(X_t) k u_k(t) d V_t, \\
u_k(0) &= \int_0^{\infty} \lambda^k \bar\mu_0( d \lambda),
\end{aligned}
\label{MomentSystemLLN}
\end{equation}
see \citeasnoun{GieseckeSpiliopoulosSowersSirigano2012}. The fluctuation moments  $\{ v_k(t) \}_{k=1}^{\infty}$ can also be shown to satisfy a system of SDEs using the stochastic evolution equation  (\ref{Eq:CLT}). Taking a test function $f=\lambda^k$ in (\ref{Eq:CLT}), we have
\begin{align}\label{MomentSystem}
\begin{aligned}
d v_{k}(t)&=\beta^{C} k u_{k-1}(t)v_{1}(t)dt+\left[0.5\sigma^{2} k(k-1)+\alpha\bar{\lambda} k+k \beta^{C}u_{1}(t)\right]v_{k-1}(t)dt-v_{k+1}(t)dt\\
&+\left[k \beta^{S} b_{0} (X_{t})-k \alpha+0.5 k (k-1)\left(\beta^{S}\sigma_{0}(X_{t})\right)^{2}\right]v_{k}(t)dt+k \beta^{S} \sigma_{0} (X_{t})v_{k}(t)dV_{t}+d \bar{\mart}_{k}(t), \\
v_k(0) &= \int_0^{\infty} \lambda^k \bar \Xi_{0}( d \lambda ).
\end{aligned}
\end{align}
where $\bar{\mart}_{k}(t)  = \la  \lambda^k, \bar{\mart}_t  \ra$ and its covariation is
\begin{align*}
\big{[} d \bar{\mart}_{k}(t), d \bar{\mart}_{j}(t) \big{]}
&= \big{(} \sigma^2 k j u_{k+j-1}(t) + u_{k+j+1} - \beta^C k u_{k-1} u_{j+1} (t) \notag \\
&- \beta^C j u_{j-1}(t) u_{k+1}(t) + (\beta^C)^2 k j u_{k-1}(t) u_{j-1}(t) u_1(t) \big{)} dt.
\end{align*}

From these covariations we form the covariation matrix $\Sigma_{\mart}(t)$.  Note that this matrix depends upon the path of $X$; the solution to the SDE system (\ref{MomentSystem})  is conditionally Gaussian given a path of $X$.

For the derivation of (\ref{MomentSystem}) to be rigorous, we must show that $\lambda^k$ belongs to the weighted Sobolev space $W_0^J(w, \rho)$ for every $k \in \mathbb{N}$. The choice of the weight functions $w$ and $\rho$ is not unique; see Section \ref{S:SobolevSpace}. We will make a choice that is convenient for our purposes but may not be minimal. Let $\rho = 1$, and $w = \exp( - c \lambda)$ for $c > 0$.  Then, $\rho^{\ell -1} D^{\ell} \rho$ and $w^{-1} \rho^{\ell} D^{\ell} w$ are bounded for every $\ell \leq J$ and Condition \ref{A:AssumptionsOnWeights} in Section \ref{S:SobolevSpace} is satisfied.  There exists a unique solution to the stochastic evolution equation for $\bar \Xi$ in this chosen space.  $W_0^J(w,p)$ is the closure of $\mathcal{C}_0^{\infty}$ in the norm $|| \cdot ||_{W_0^J(w, \rho) }$.  Therefore, to prove $\lambda^k$ belongs to $W_0^J(w, \rho)$, we show that a sequence $f_m^k \in \mathcal{C}_0^{\infty}$ approaches $\lambda^k$ as $m \to\infty$ under the norm $|| \cdot ||_{W_0^J(w, \rho) }$.

More precisely, let $f_m^k = \lambda^k B_m(\lambda)$ where $B_m(\lambda) =1 $ on $[0, K_m]$, $0 \leq B_m \leq 1$ on $(K_m, K_m + \epsilon)$, and $B_m(\lambda) = 0$ on $[K_m + \epsilon, \infty)$. Here, $(K_m)$ is a sequence of numbers tending to $\infty$ and $\epsilon>0$. Furthermore, $B_m( \lambda)$ is smooth.  Such a function is
\begin{align*}
 B_m( \lambda) &= 1 - \frac{\int_{-\infty}^{\lambda} g_m(x) dx }{ \int_{-\infty}^{\infty} g_m(x) dx }, \notag \\
 g_m( x) &= h(x-K_m) h(K_m+ \epsilon-x), \notag
\end{align*}
where $h(x) = e^{- \frac{1}{x}}$ when $x > 0$ and $0$ for $x \leq 0$.  Then, $f_m^k$ is a smooth function with compact support.  We also note that the integral $\int_{-\infty}^{\infty} g_m(x) dx$ only depends upon $\epsilon$; it is not affected by $K_m$.  To prove that $|| \lambda^k - f_m^k(\lambda) ||_{W_0^J(w, \rho) } \to 0$ as $K_m \to\infty$, one has to show that the following integral tends to zero as $K_m \to \infty$ for each $\ell \leq J$.:
\begin{align*}
\int_0^{\infty} e^{-c \lambda} |D^{\ell} [\lambda^k - f_m^k(\lambda) ] |^2 d \lambda &= \int_0^{K_m} e^{-c  \lambda} |D^{\ell} [\lambda^k -\lambda^k B_m(\lambda) ] |^2 d \lambda + \int_{K_m}^{K_m+\epsilon} e^{-c \lambda} |D^{\ell} [\lambda^k -\lambda^k B_m(\lambda) ] |^2 d \lambda \notag \\
&+ \int_{K_m+\epsilon}^{\infty} e^{-c \lambda} |D^{\ell} [\lambda^k -\lambda^k B_m(\lambda) ] |^2 d \lambda.
\end{align*}
The first term trivially is zero due to the definition of $B_m$ while the third term can be shown to tend to zero using integration by parts.  Turning to the second integral, we have that
\begin{align*}
\int_{K_m}^{K_m+\epsilon} e^{-c \lambda} |D^{\ell} [\lambda^k -\lambda^k B_m(\lambda) ] |^2 d \lambda \leq
 \int_{K_m}^{K_m+\epsilon} e^{-c \lambda} |D^{\ell} \lambda^k |^2 +  \int_{K_m}^{K_m+\epsilon} e^{-c \lambda} |D^{\ell} [\lambda^k B_m(\lambda) ] |^2  d \lambda.
\end{align*}
The first integral again obviously tends to zero by integration by parts.  In order to show that the second term tends to zero, one has to demonstrate that the $ |D^{\ell} B_m(\lambda)  |^2 $ is uniformly bounded in $m$ on $[K_m, K_m+ \epsilon]$.  For $\ell = 0$, the result is trivial.  For $\ell = 1$, we have that
\begin{align*}
D^{1}  B_m(\lambda)  =  C \exp \left(- \frac{\epsilon}{y (\epsilon - y) }  \right) \leq C \exp \left( - \frac{4}{\epsilon} \right),
\end{align*}
where $y = \lambda - K_m$ and $y \in [0, \epsilon ]$.  For $\ell =2$, we have that
\begin{align*}
D^{2}  B_m(\lambda)  =  C \exp \left(- \frac{\epsilon}{y (\epsilon - y) }  \right) \left( \frac{1}{(\epsilon - y)^2} - \frac{1}{y^2} \right).
\end{align*}
Exponential decay dominates the polynomial growth, so  $D^{2}  B_m(\lambda)$ and its derivative approach zero as $y$ approaches $0$ or $\epsilon$ .  Since $[0, \epsilon]$ is a bounded domain, $D^{2}  B_m(\lambda)$ is therefore also bounded on $[0, \epsilon]$.  Higher derivatives can be treated similarly. Therefore, since $ |D^{\ell} B_m(\lambda)  |^2 $ is uniformly bounded on $[K_m, K_m + \epsilon]$, the second integral also tends to zero.

Unless the volatility parameter $\sigma=0$, the SDE system (\ref{MomentSystem}) is not closed.\footnote{If $\sigma=0$, the LLN reduces to two SDEs and the fluctuation limit also reduces to two SDEs.}  For computational purposes, it must be truncated at some level $K$.   Note that this means one also needs the LLN moments $\{ u_k \}_{k=0}^{2 K+1}$.  
The system (\ref{MomentSystemLLN}) must also be truncated and can be solved using an Euler scheme.  Important details for the numerical solution of (\ref{MomentSystemLLN}) as well as an alternative formulation as a random ODE system are described in \citeasnoun{GieseckeSpiliopoulosSowersSirigano2012}.

\subsection{Case with no Systematic Risk}
When $\beta^S = 0$, there exists a semi-analytic solution. Observe that the moment system (\ref{MomentSystemLLN}) is now a system of ODEs and that (\ref{MomentSystem}) is a linear system with a Gaussian forcing term. Therefore, the system of moments $\mathbf{v}(t)$ is itself Gaussian. One can directly compute its distribution. We rewrite (\ref{MomentSystem}) concisely as
\begin{eqnarray}
d \mathbf{v}(t) = A(t) \mathbf{v}(t) dt + d \bar{\mathbf{\mart}}(t), \quad \mathbf{v}(0) = \mathbf{v}_0,
\label{GaussODEOne}
\end{eqnarray}
where $A: [0,T] \mapsto \mathbb{R}^{K+1, K+1}$ and $\mathbf{v}, \bar{\mathbf{\mart}}: [0,T] \times \Omega \mapsto \mathbb{R}^{K+1}$.  Now, let $\Psi: [0,T] \mapsto \mathbb{R}^{K+1, K+1}$ be the fundamental solution matrix satisfying
\begin{eqnarray}
d \Psi(t) = A(t) \Psi(t) dt, \quad \Psi(0) = I,
\label{FundamentalSolution}
\end{eqnarray}
where $I$ is the identity matrix.  If $\beta^C = 0$, $A$ is a constant matrix and there is an exponential matrix solution.  For $\beta^C >0$, one can either solve (\ref{FundamentalSolution}) numerically using a method such as Runge Kutta or the Magnus series approximation.  We assume that $\mathbf{v}_0$ is deterministic (if it is a Gaussian random variable, the result is very similar).  Then,
\begin{eqnarray*}
\mathbf{v}(t) = \Psi(t) \mathbf{v}_0 + \Psi(t) \int_0^t \Psi^{-1}(s) d \bar{\mathbf{\mart}}(s).
\end{eqnarray*}
It easily follows that $\mathbf{v}(t) \sim \mathcal{N}( \Psi(t) \mathbf{v}_0, \Sigma(t))$ where
\begin{eqnarray}
\Sigma(t) = \Psi(t) \big{[} \int_0^t  \Psi^{-1}(s) \Sigma_{\mart}(s) \big{(} \Psi^{-1}(s) \big{)}^{\top} ds \big{]} \Psi(t)^{\top}.
\label{SemiAnalyticVariance}
\end{eqnarray}
Therefore, we have avoided simulation of (\ref{GaussODEOne}) by instead developing a semi-analytic approach for the computation of the distribution.  Also, note that in the case $\beta^C = \beta^S = 0$, we have a closed-form formula for the approximating loss since $\Psi(t) = e^{A t}$.

Figure \ref{fig1} shows a comparison between the approximate loss distribution according to (\ref{ApproxMain}) and the actual loss distribution for a pool of $N=1,000$ names, for each of several values of the contagion sensitivity $\beta^C$. Here and in the numerical experiments described below, the actual distribution is estimated by simulating the default times in (\ref{E:main}) using the discretization method detailed in \citeasnoun{GieseckeSpiliopoulosSowersSirigano2012}.  The second-order Gaussian approximation (\ref{ApproxMain}) is a significant improvement over the first-order approximation (\ref{firstorder}) implied by the LLN.   The latter produces a delta function for this case while the Gaussian approximation is able to accurately capture a large portion of the remaining noise in the finite system.
\begin{figure}[t]
\begin{center}
\includegraphics[scale=0.8]{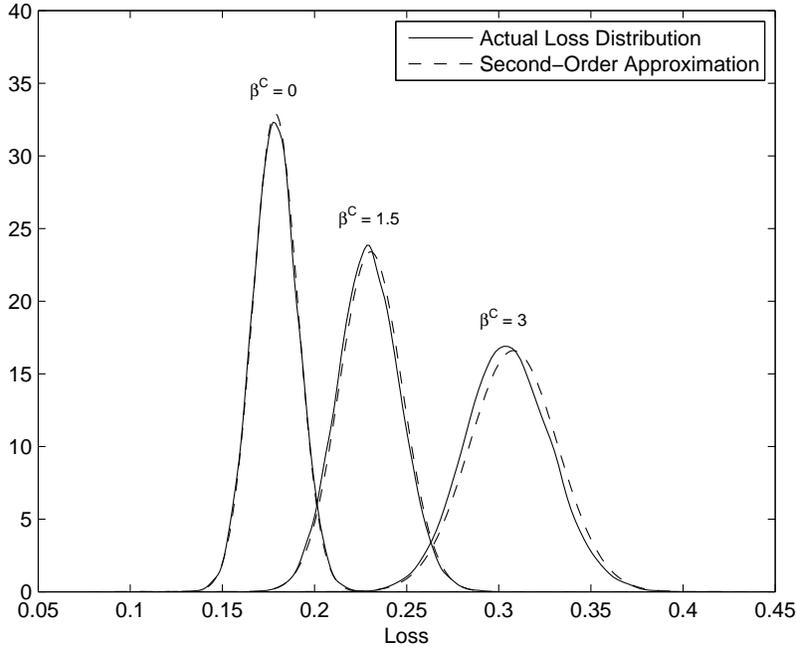}
\caption{Comparison of approximate loss distribution and actual loss distribution in the finite system at $T = 1$ for $N = 1,000$.  The parameter case is $\sigma = .9, \alpha = 4, \lambda_0 = .2, \bar{\lambda} = .2$, and $ \beta^S = 0$.  }
\label{fig1}
\end{center}
\end{figure}

The linear stochastic evolution equation (\ref{Eq:CLT}) involves a linearization of the jump term $\beta^C d L_t^N$ in (\ref{E:main}).  It is reasonable to expect that the accuracy of the approximation will decrease as $\beta^C$ increases.  This is confirmed by the numerical results.  For larger pools, the linearization will become more accurate.  In the case of $\beta^C = 3$, we are taking too small of a pool.

Another common concern with linearizations is a loss of accuracy over long time horizons.  However, we do not observe any significant loss of accuracy for the fluctuation limit in our model.

\subsection{Case with Systematic Risk}
In the case with systematic risk (i.e., $\beta^S > 0$), the solution to  (\ref{MomentSystem}) is conditionally Gaussian given a path of $X$.  There are two approaches to treat the general case.  We present the most obvious approach first and then a second scheme which achieves lower variance by replacing a portion of the simulation with a semi-analytic computation.

\subsubsection{Scheme 1: Direct Simulation} \label{DirectSim}
Here we directly discretize and solve the systems  (\ref{MomentSystemLLN})  and (\ref{MomentSystem}) (for example, using an Euler method).  Alternatively, one could instead numerically solve the random ODE formulation of (\ref{MomentSystemLLN}) given in \citeasnoun{GieseckeSpiliopoulosSowersSirigano2012} and then solve (\ref{MomentSystem}).
\begin{itemize}
\item Simulate paths $X^1, \ldots, X^M$ of the systematic risk process on $[0,T]$.
\item Conditional upon a path $X^m$, the martingale term $\mathbf{\bar{\mart}}$ is Gaussian.
\begin{itemize}
\item  First, solve for the LLN moments $\mathbf{u}^m(t)$ in (\ref{MomentSystemLLN}) for a path $X^m$ and then calculate the conditional covariation matrix $\Sigma_{\mart}^m(t)$.
\item Perform a spectral decomposition of the covariation matrix $\Sigma_{\mart}^m(t)$ at a discrete set of times on a grid $\mathcal{T}$.
\item For $j = 1, \ldots, J$, discretize system (\ref{MomentSystem}) on the grid $\mathcal{T}$.  At each time $t_i \in \mathcal{T}$, draw a sample $d \mathbf{\bar{\mart}}^{m,j}(t_i)$ using the spectral decomposition of $\Sigma_{\mart}^m(t)$ and then solve (\ref{MomentSystem}) using the Euler method.  This produces $J$ samples from the conditionally Gaussian solution given the path $X^m$.
\item  Finally, for each $j$, compute a sample of the ``approximate conditional loss'' $L_t^{m,j,N}$ for a pool of size $N$ as the difference between the conditional LLN loss $L^m_t=1-u^m_0(t)$ and a sample of the conditionally Gaussian solution scaled by $1/\sqrt{N}$. 
\end{itemize}
\item Approximate the unconditional distribution of $L^N_t$ by $\frac{1}{J M} \sum_{m=1}^M \sum_{j=1}^J \delta_{L_t^{m,j,N}}$.
\end{itemize}
Let $\sigma_1^2 = \textrm{Var} [\int_0^{\infty} f(y) p_t(y|X) dy]$ and $\sigma_2^2 = \mathbb{E}[ \textrm{Var}[ f(L_t^N)| X] ] $ where $p_t(y|X)$ is the conditional density of the approximate loss given $X$ at time $t$.  The total simulation time is $M \tau_1 + M J \tau_2$, where  $\tau_1$ is the time needed to simulate the systematic risk process and solve the LLN moment system while $\tau_2$ is the time required to solve the fluctuation moment system.  If we want to estimate the expectation of a function $f$ of the approximate loss at some time $t$, it is straightforward to show using the first order condition for a minimum that the optimal allocation of simulation resources $M^{\ast}$ and $J^{\ast}$ to minimize the estimator's variance for a fixed computational time $\tau$ is
\begin{align}
M^{\ast} & \approx \frac{-2  \tau_1 \sigma_1^2 + \sqrt{ (2  \tau_1 \sigma_1^2)^2 + 4 \sigma_1^2  ( \sigma_2^2 \tau_1 \tau_2 - \sigma_1^2 \tau_1^2) } }{ 2 ( \sigma_2^2 \tau_1 \tau_2 - \sigma_1^2 \tau_1^2) } \tau = M_0^{\ast} \tau \notag \\
J^{\ast} & \approx \frac{1- M_0^{\ast} \tau_1}{ \tau_2 M_0^{\ast}}.
\label{OptimalChoice}
\end{align}
The optimal number for $J$ does not depend upon the total computational resources.

\subsubsection{Scheme 2: Conditionally Semi-Analytic Approach} \label{SemiScheme}
We propose an alternate scheme that replaces some of the Monte Carlo simulation of the previous scheme with a semi-analytic calculation.  First, we rewrite (\ref{MomentSystem}) concisely as
\begin{align}
d \mathbf{v}(t) = A(t) \mathbf{v}(t) dt + B \mathbf{v}(t) dX_t + d \mathbf{\bar{\mart}}(t), \quad \mathbf{v}(0) = \mathbf{v}_0,
\label{LinearGaussTwo}
\end{align}
where $A(t)$ is a random, time-dependent matrix and $B$ is a constant matrix.  In the calculations that follow, we assume that $\mathbf{v}_0$ is deterministic; the computations when $\mathbf{v}_0$ is a Gaussian random variable are similar.  The fundamental solution $\Psi: [0,T] \times \Omega \mapsto \mathbb{R}^{K+1, K+1}$ satisfies
\begin{align}
d \Psi(t) = A(t) \Psi(t) dt + B \Psi(t) d X_t, \quad \Psi(0) = I.
\label{fundamentalX}
\end{align}
Then, since $[dX_t, d \bar{\mart}_k (t)] =0$ for any $k \geq 0$,
\begin{align}
\mathbf{v}(t) = \Psi(t) \mathbf{v}_0 + \Psi(t) \int_0^t \Psi^{-1}(s) d \mathbf{\bar{\mart}}(s) .
\end{align}
(It is easy to show that this is a solution to (\ref{LinearGaussTwo}) by substituting it back into that SDE.)  Unfortunately, even when $\beta^C = 0$, $b_0(x) = b$,  and $\sigma_0(x) = \sigma_0$ so that $A$ is a constant matrix, there is no exponential solution to (\ref{fundamentalX}) since $A$ and $B$ do not commute.

To calculate the approximation (\ref{ApproxMain}), one must simulate from the joint law of the paths of $X$ and $\mathbf{\bar{\mart}}$ over $[0,T]$.  We can do this using the following scheme:
\begin{itemize}
\item Simulate paths $X^1, \ldots, X^M$ of the systematic risk process on $[0,T]$.
\item Conditional upon a path $X^m$, we have a Gaussian solution for the fluctuation moments $\mathbf{v}^m(t)$.  Furthermore, we can semi-analytically compute the conditional distribution of  $\mathbf{v}^m(t)$.
\begin{itemize}
\item Given a path $X^m$, solve for the LLN moments  $\mathbf{u}^m(t)$ (which also yield the conditional LLN loss $L^m$) and the fundamental solution $\Psi^m(t)$.
  \item Then, form the conditional covariation matrix $\Sigma_{\mart}^m(t)$.  Finally, compute the conditional covariance matrix using the closed-form formula (\ref{SemiAnalyticVariance}).  This yields $\textrm{Var}[ v_0^m(t)]$.
\end{itemize}
\item Approximate the unconditional distribution of $L^N_t$ by $\frac{1}{M} \sum_{m=1}^M \mathbb{P}^m_t$, where $\mathbb{P}^m_t$ is a Gaussian measure with mean $L^m$ and variance $\textrm{Var}[ v_0^m(t)]/N$.
\end{itemize}

A skeleton of the systematic risk process $X$ for both schemes discussed above can be generated exactly (without discretization bias) using the methods of \citeasnoun{Beskos}, \citeasnoun{chen}, or \citeasnoun{smelov}.  It is also worthwhile to highlight that either scheme yields an approximation to the distribution of the loss $L^N_t$ for all time horizons $t \leq T$ (i.e., the ``loss surface'')  and all portfolio sizes $N \in \mathbb{N}$ simultaneously.

Scheme 2 will have lower variance than Scheme 1.  Let $\hat{f}^1$ and $\hat{f}^2$ be estimators of $\mathbb{E} [ f(L_t^N)]$ using Schemes $1$ and $2$, respectively.  We have
\begin{align*}
 \textrm{Var} [ \hat{f}^1 ] = \textrm{Var} \bigg[ \frac{1}{J M} \sum_{m=1}^M \sum_{j=1}^J f(L^{m,j,N}_t) \bigg] =  \frac{1}{M} \big{(} \sigma_1^2  + \frac{1}{J} \sigma_2^2 \big{)} \notag > \textrm{Var} [ \hat{f}^2 ] = \frac{1}{M} \sigma_1^2.
\end{align*}
$\textrm{Var} [ \hat{f}^2 ] = \frac{1}{M} \sigma_1^2$ since Scheme $2$ generates samples from the random variable $\int_0^{\infty} f(y) p_t(y|X) dy$.

We note that there is numerical instability for large $\beta^S$ for both schemes, especially over long time horizons.  One must use a small time-step to avoid this instability.  The instability is caused by the exponential growth terms $\frac{1}{2} ( \beta^S \sigma_0(X_t) )^2 k (k-1) v_k(t)$ and $\frac{1}{2} ( \beta^S \sigma_0(X_t) )^2 k (k-1) u_k(t)$.  When one is only interested in calculating the LLN approximation (\ref{firstorder}) via the moment system (\ref{MomentSystemLLN}), the following transformed moments significantly reduce instability by removing the exponential growth term:
\begin{eqnarray*}
w_k(t) = \exp\left(- \frac{1}{2} (\beta^S)^2 k (k-1) \int_0^t \sigma_0(X_s)^2 ds \right) u_k(t).
\end{eqnarray*}
However, when interested in solving (\ref{MomentSystemLLN}) and (\ref{MomentSystem}) in conjunction, there is no simple transformation.  The best approach is to solve (\ref{MomentSystem}) with a sufficiently small time step such that it is stable and then solve for the transformed fluctuation moments $\tilde{w}_k(t) = \exp(- \frac{1}{2} (\beta^S)^2 k (k-1) \int_0^t \sigma_0(X_s)^2 ds) v_k(t)$ with a larger time step.

Figure \ref{figTruncation} compares the approximate loss distribution according to (\ref{ApproxMain}) for different truncation levels $K$ of the fluctuation moment system (\ref{MomentSystem}).  We use a time step of $0.005$ and produce samples from $X$ using an Euler scheme.  The approximate loss distribution converges very rapidly in terms of the truncation level.  This conforms with previous numerical studies of the LLN moment system (\ref{MomentSystemLLN}) which also demonstrated its quick convergence rate; see \citeasnoun{GieseckeSpiliopoulosSowersSirigano2012}.

\begin{figure}[t!]
\begin{center}
\includegraphics[scale=0.7]{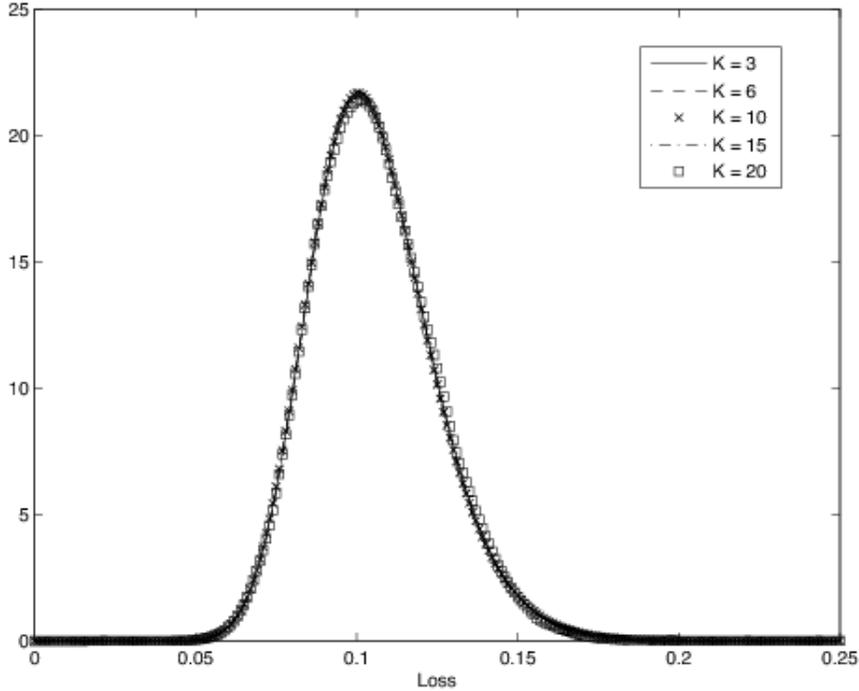}
\caption{Comparison of approximate loss distribution for different truncation levels at $T = 0.5$. Parameter case is $\sigma = 0.9, \alpha = 4, \lambda_0 = \bar{\lambda} = 0.2, \beta^C = 1,$ and $ \beta^S = 1$.  The systematic risk $X$ is an OU process with mean $1$, reversion speed $2$, volatility $1$, and initial value $1$.  Approximate loss distribution computed using the conditionally semi-analytic Scheme 2.  The fluctuation moments are truncated at level $K$ and the LLN moments are truncated at level $3 K$.}
\label{figTruncation}
\end{center}
\end{figure}

We now study the validity of the second-order approximation (\ref{ApproxMain}) by comparing it against the actual loss distribution from simulating the true finite system.  In the following numerical studies, we choose a time step of $0.005$ and a truncation level of $K = 6$ for both schemes described above.  A time step of $0.005$ is used for simulating the finite system.  Samples from $X$ are produced using an Euler scheme.  Figure \ref{fig2} compares the approximate loss distribution according to (\ref{ApproxMain}) with the actual loss distribution for $\beta^C = 0$ and $\beta^S = 1$, for each of several portfolio sizes.  The approximate loss distribution is extremely accurate, even for a very small pool with only $N = 250$.  We also observe that the approximation accurately captures the tails of the actual loss distribution. Figure \ref{fig3} compares the approximate loss distribution with the actual loss distribution for $\beta^C =1$ and $\beta^S = 1$.  The approximate loss distribution is again accurate, although not as accurate as when $\beta^C = 0$.  We also show in both Figures \ref{fig2} and \ref{fig3} the first-order LLN approximation (\ref{firstorder}).  It is clear that the second-order approximation has increased accuracy, especially for smaller portfolios and in the tail of the distribution.  Finally, Figure \ref{VaR} shows a comparison for the $95$ and $99$ percent value at risk (VaR) between the actual loss, LLN approximation (\ref{firstorder}), and approximation (\ref{ApproxMain}) for a pool of $N = 1,000$.  The approximation for the VaR based on (\ref{ApproxMain}) is significantly more accurate than the LLN approximation for the VaR.

\begin{figure}[t]
\begin{center}
\includegraphics[scale=0.8]{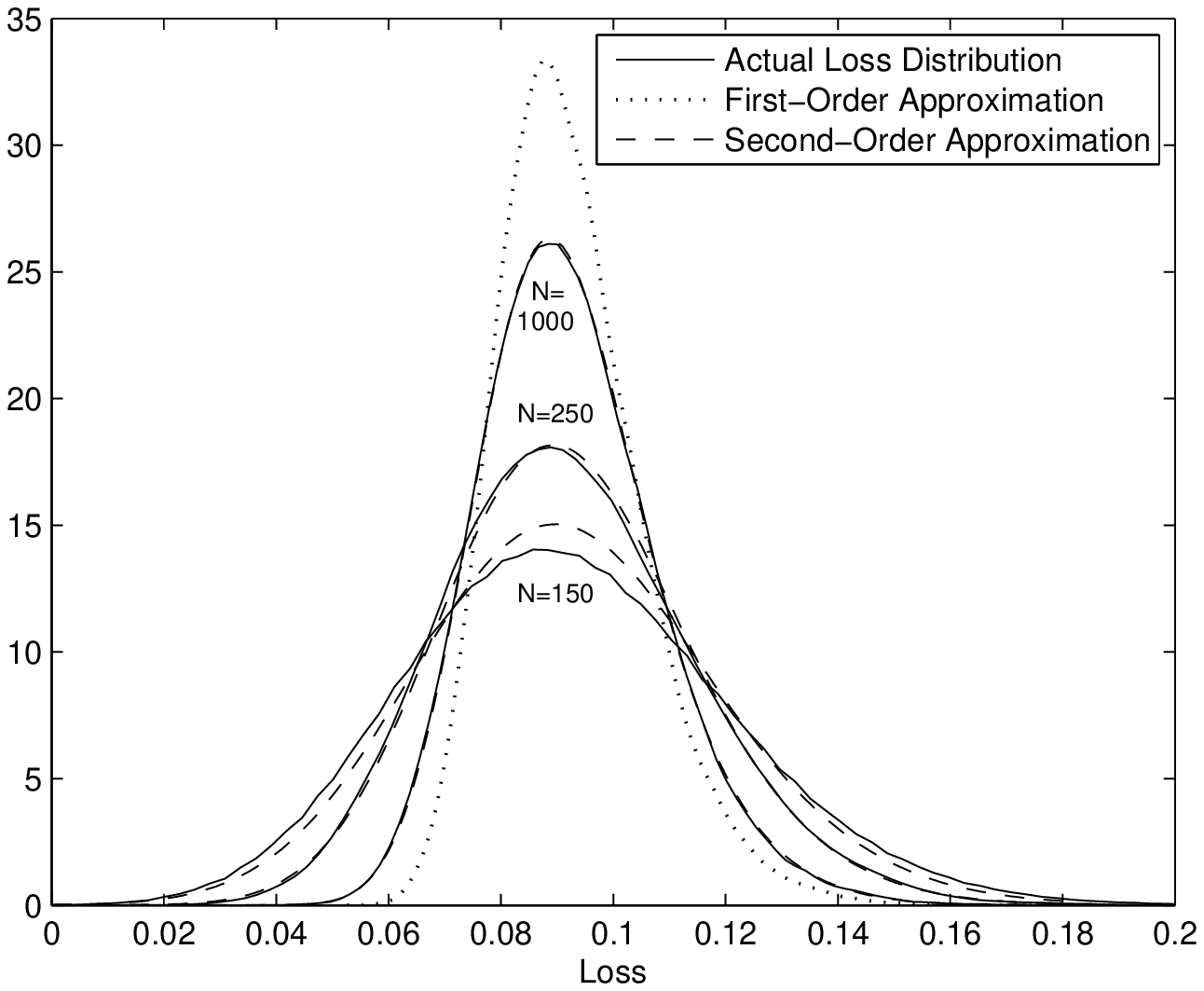}
\caption{Comparison of approximate loss distribution and actual loss distribution in the finite system at $T = 0.5$. Parameter case is $\sigma = 0.9, \alpha = 4, \lambda_0 = \bar{\lambda} = 0.2, \beta^C = 0$, and $ \beta^S = 1$.  The systematic risk $X$ is an OU process with mean $1$, reversion speed $2$, volatility $1$, and initial value $1$.  Approximate loss distribution computed using the conditionally semi-analytic Scheme 2.}
\label{fig2}
\end{center}
\end{figure}

\begin{figure}[t]
\begin{center}
\includegraphics[scale=0.8]{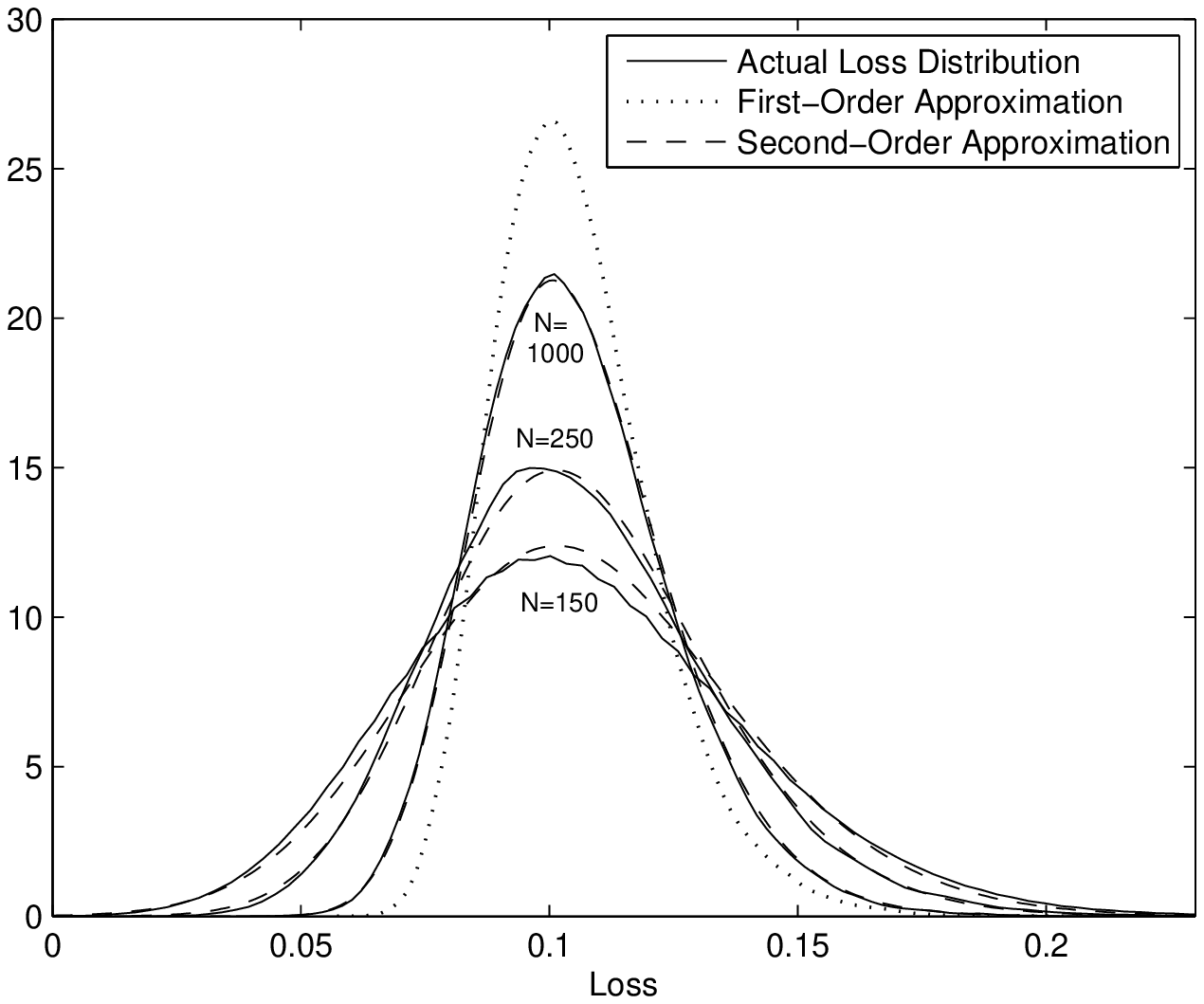}
\caption{Comparison of approximate loss distribution and actual loss distribution in the finite system at $T = 0.5$. Parameter case is $\sigma = 0.9, \alpha = 4, \lambda_0 = \bar{\lambda} = 0.2, \beta^C = 1,$ and $ \beta^S = 1$.  The systematic risk $X$ is an OU process with mean $1$, reversion speed $2$, volatility $1$, and initial value $1$.  Approximate loss distribution computed using the conditionally semi-analytic Scheme 2.}
\label{fig3}
\end{center}
\end{figure}

\begin{figure}[t]
\begin{center}
\includegraphics[scale=0.8]{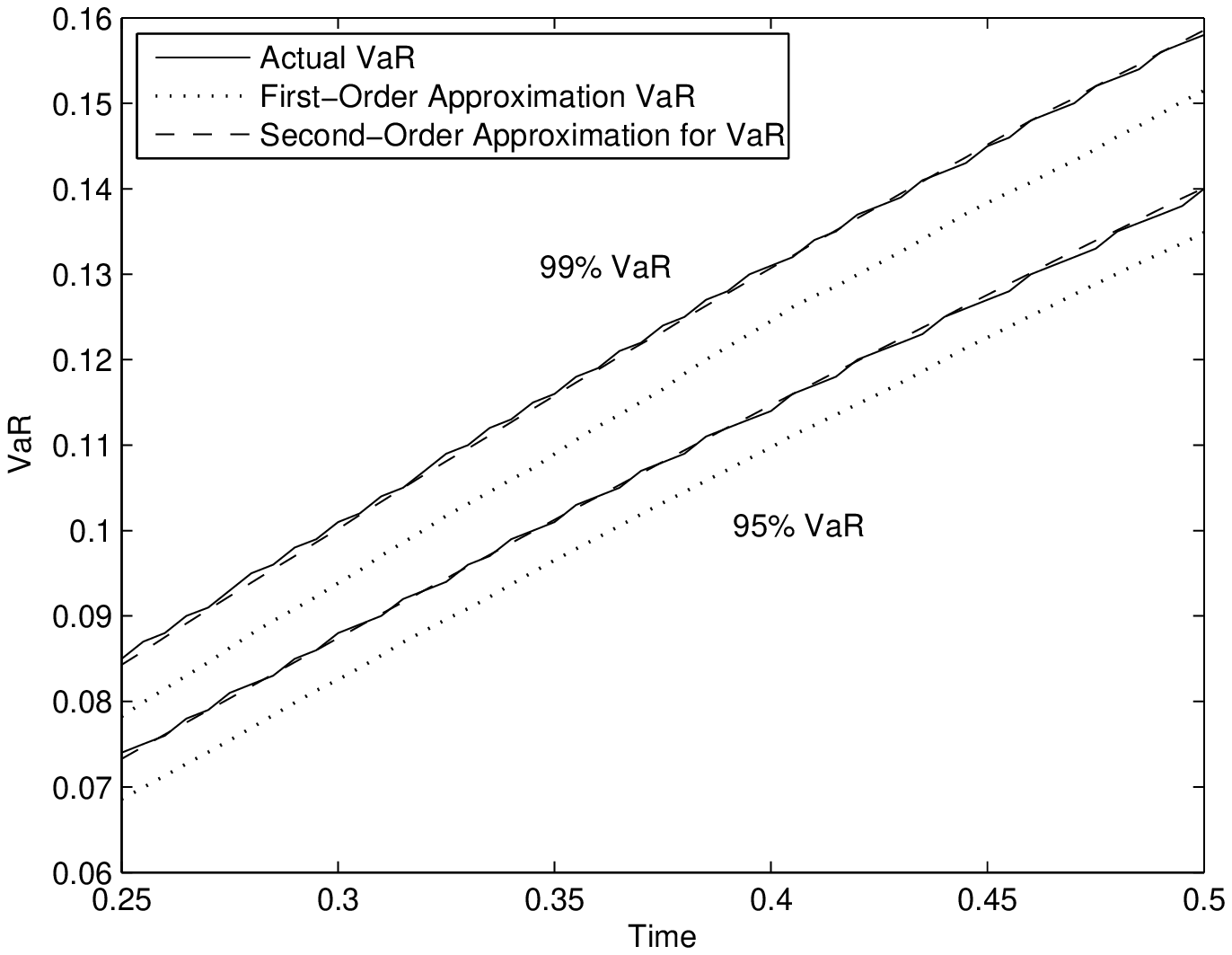}
\caption{Comparison of approximate and actual value at risks for $N = 1,000$. Parameter case is $\sigma = 0.9, \alpha = 4, \lambda_0 = \bar{\lambda} = 0.2, \beta^C = 1,$ and $ \beta^S = 1$. The systematic risk $X$ is an OU process with mean $1$, reversion speed $2$, volatility $1$, and initial value $1$. Approximate loss distribution computed using the conditionally semi-analytic Scheme 2.}
\label{VaR}
\end{center}
\end{figure}

Appendix \ref{app} provides additional numerical results on the performance of the numerical schemes.  In particular, we demonstrate the advantages of the second-order approximation over direct simulation of the finite system (\ref{E:main}).  For a fixed computational time, the standard error for the second-order approximation is several orders of magnitude smaller than the standard error for the finite system simulation.

\subsection{Approximate Loss Process}
One of the significant advantages of the fluctuation analysis presented here is the dynamic approximation it provides for the entire loss {\it process}.  An approximation of the loss process may be used to estimate the probability that the loss per month never exceeds a certain amount, or the density of the hitting time for the loss exceeding a certain level.

One can generate a sample skeleton of the approximate loss process according to the right hand side of (\ref{ApproxMain}) using Scheme 1. Although Scheme 2 gives the loss distribution at all time horizons, it does not directly yield a skeleton of the approximate loss process. However, due to the process being conditionally Gaussian, one can use a slightly modified form of Scheme 2 to simulate such a skeleton on as fine a time grid as desired. For small $\beta^S$, one can expect that Scheme $2$ will be more computationally efficient than Scheme $1$ (see Appendix \ref{app} for numerical results).

To simulate skeletons for the moments (\ref{MomentSystem}) of the fluctuation limit, one can use the following method.  Let $\mathbf{v}_0 = 0$ (the general case is very similar).   As usual, begin by simulating a path from the systematic risk process $X$.  Given a path of $X$, we then simulate the loss at $t = T$ from the conditionally Gaussian distribution given by the formula (\ref{SemiAnalyticVariance}).  To simulate at $0<s<T$, we only need the covariance matrix $\textrm{Cov} [ \mathbf{v}_T, \mathbf{v}_s | \mathcal{V}_{T}]$.  One can continue to simulate at grid points of one's choice.  For instance, one can simulate at time $s_1$ where  $s < s_1 < T$ given the covariance matrices  $\textrm{Cov} [ \mathbf{v}_T, \mathbf{v}_s | \mathcal{V}_{T}]$, $\textrm{Cov} [ \mathbf{v}_T, \mathbf{v}_{s_1} | \mathcal{V}_{T} ]$, and $\textrm{Cov} [ \mathbf{v}_s, \mathbf{v}_{s_1} | \mathcal{V}_{T}]$.  These covariance matrices are easily computed using the fundamental solution (\ref{fundamentalX}):
\begin{align*}
\Sigma(\tau_1, \tau_2)  = \textrm{Cov} [ \mathbf{v}_{\tau_1}, \mathbf{v}_{\tau_2} | \mathcal{V}_{T} ] = \Psi(\tau_1) \big{[} \int_0^{\tau_1 \wedge \tau_2}  \Psi^{-1}(s) \Sigma_{\mart}(s) \big{(} \Psi^{-1}(s) \big{)}^{\top} ds \big{]} \Psi(\tau_2)^{\top}.
\end{align*}
To simulate $\mathbf{v}_s$ given $\mathbf{v}_{\tau_1}$ and $\mathbf{v}_{\tau_2}$ where $\tau_1 < s < \tau_2$, we generate a sample from the multivariate Gaussian distribution $\mathcal{N}(\tilde{\mu}, \tilde{\Sigma})$ where $\tilde{\mu} = \Sigma_1 \Sigma_2^{-1} \mathbf{Z}$ and $\tilde{\Sigma} = \Sigma(s,s) - \Sigma_1 \Sigma_2^{-1} \Sigma_1$. The matrices $\Sigma_1, \Sigma_2,$ and $\mathbf{Z}$ are
\begin{align*}
\Sigma_1&=
 \begin{bmatrix}
  \Sigma(s, \tau_1) & \Sigma(s, \tau_2)
 \end{bmatrix}
\\
\Sigma_2&=
 \begin{bmatrix}
  \Sigma(\tau_1, \tau_1) & \Sigma(\tau_1, \tau_2)   \\
\Sigma(\tau_1, \tau_2)  & \Sigma(\tau_2, \tau_2)
 \end{bmatrix}
\\
\mathbf{Z}&=
 \begin{bmatrix}
 \mathbf{v}_{\tau_1} \\
\mathbf{v}_{\tau_2}
 \end{bmatrix}
\end{align*}

It is worthwhile to note that this scheme requires only a single computation of the fundamental solution $\Psi^m(t)$ for each path $X^m$.  Given $\Psi^m(t)$ for $0 \leq t \leq T$, one can generate as many skeletons of the fluctuation moments as desired for that particular $X^m$ via the aforementioned method. In contrast, Scheme 1 requires one to resolve the system of SDEs (\ref{MomentSystem}) for each new realization.  In particular, this suggests that this scheme will be more efficient than Scheme $1$ for small $\beta^S$.

Once skeletons from the fluctuation limit have been simulated, skeletons of the approximate loss process can be produced by combining these fluctuation paths with the law of large numbers limit process given by (\ref{MomentSystemLLN}) and using the approximation formula (\ref{ApproxMain}).

\subsection{Nonhomogeneous Portfolio}
The moment method can be extended to a nonhomogeneous portfolio.  Define the moments $v_k(t,  \pp) = \int_{\mathbb{R}^{+}} \lambda^k \bar \Xi_t(d \lambda, \pp) $ and $u_k(t,  \pp) = \int_{\mathbb{R}^{+}} \lambda^k \bar\mu_t(d \lambda,\pp)$. Then
\begin{align*}
d v_{k}(t,  \pp)&=\beta^{C} k u_{k-1}(s,  \pp) \left( \int_{\PP} v_{1}(t, \pp) d \pp  \right) dt+\left[\frac{1}{2}\sigma^{2} k(k-1)+\alpha\bar{\lambda} k+k \beta^{C} \int_{\PP} u_{1}(s, \pp) d \pp \right]v_{k-1}(t, \pp)dt\nonumber\\
&+v_{k}(t, \pp)\left[-k \alpha+k \beta^{S} b_{0} (X_{t})+\frac{1}{2} k (k-1)\left(\beta^{S}\sigma_{0}(X_{t})\right)^{2}\right]dt
-v_{k+1}(t, \pp)dt\nonumber\\
&+k \beta^{S} \sigma_{0} (X_{t})v_{k}(t, \pp)dV_{t}+\bar{\mart}_{k}(t, \pp), \notag \\
v_k(0, \pp) &= \int_0^{\infty} \lambda^k \bar \Xi_{0}(d \lambda, \pp) .
\end{align*}
The covariation $\big{[} d \bar{\mart}_{k}(t, \pp_1), d \bar{\mart}_{j}(t, \pp_2) \big{]}$ for $\pp_1, \pp_2$ $\in$ $\PP$ is given by
\begin{align*}
\big{[} d \bar{\mart}_{k}(t, \pp_1), d \bar{\mart}_{j}(t, \pp_2) \big{]} &= \mathbf{1}_{\pp_1 = \pp_2} \big{(} \sigma^2 k j u_{k+j-1}(t, \pp_1) + u_{k+j+1}(t, \pp_1)
- \beta^C k u_{k-1}(t, \pp_1) u_{j+1} (t, \pp_1) \notag \\
&- \beta^C j u_{j-1}(t, \pp_1) u_{k+1}(t, \pp_1) + (\beta^C)^2 k j u_{k-1}(t, \pp_1) u_{j-1}(t, \pp_1) u_1(t, \pp_1) \big{)} dt.
\end{align*}
The LLN moments are
\begin{align*}
d u_k(t, \pp) &= u_k(t, \pp) \left[-  \alpha  k+ \beta^S b_0(X_t) k +\frac{1}{2} (\beta^S)^2 \sigma_0^2(X_t) k(k-1)\right] dt - u_{k+1}(t, \pp) dt \\
&+ u_{k-1}(t, \pp) \left[ 0.5 \sigma^2 k (k-1) + \alpha \bar{\lambda} k + \beta^C k \int_{\PP} u_1(t, \pp) d \pp \right] dt  + \beta^S \sigma_0(X_t) k u_k (t, \pp) d V_t\\
u_k(0, \pp) &= \int_0^{\infty} \lambda^k \bar\mu_0(d \lambda,\pp).
\end{align*}

For numerical implementation, one must discretize the parameter space $\PP$.  As we allow more parameters to vary, $\PP$ can become high dimensional.  This can become computationally expensive.

\section{Proof of Theorem \ref{T:MainCLT}}\label{S:ProofMainCLT}
In this section we prove the fluctuation Theorem \ref{T:MainCLT}. The methodology of the proof goes as follows. After some  preliminary computations,
we obtain a convenient formulation of the equation that $\{\Xi^{N}_{t}\}$ satisfies, see (\ref{Eq:PrelimitEquation}). Some terms in this equation will vanish in the limit as
$N\rightarrow\infty$; this is Lemma \ref{L:AuxilliaryBounds}. Based on tightness of the involved processes (the topic of Section \ref{S:Tightness}) and continuity properties of the operators involved, we can then pass to the limit
and thus identify the limiting equation. The limit process satisfies in the weak form another stochastic evolution, which has a unique solution (proven in Section \ref{S:Uniqueness}).
The difficulty is to identify a rich enough space where tightness and uniqueness can be simultaneously proven. This is not trivial in our case, because the coefficients are not bounded
and the equation degenerates (coefficients of the highest derivatives are not bounded away from zero). It turns out that the appopriate space in which one can prove both tightness and uniqueness
is a weighted Sobolev space, introduced in \citeasnoun{Purtukhia1984} and further generalized in  \citeasnoun{GyongiKrylov1992} to study stochastic partial differential equations with unbounded coefficients,
and are briefly reviewed in Section \ref{S:SobolevSpace}.

An application of It\^{o}'s formula shows that for $f\in C^{2}_{b}(\hat \PP)$,
\begin{align}
\la f, \mu^{N}_t\ra_E &= \la f, \mu^{N}_0\ra_E +\int_{0}^{t}\left\{\la \genL_1f, \mu^{N}_s\ra_E+ \la \QQ, \mu^{N}_s\ra_E
\la \genL_2f, \mu^{N}_s\ra_E+\la \genL^{X_{s}}_3 f,
\mu^{N}_s\ra_E\right\}ds\nonumber\\
&+\int_{0}^{t}\hat{A}^{N}[f](X_{s},\mu^{N}_{s})ds+ \int_{0}^{t}\la \genL_4^{X_{s}}f,\mu^{N}_{s}\ra_E dV_{s}+\la f, \mart^{N}_t\ra\label{Eq:PrelimitMu}
\end{align}
where
\begin{equation*}\label{E:limgen3}
\hat{\genA}^{N}[f](X_{s},\mu^{N}_{s}) = \sum_{n=1}^N \lambda^\NN_s\jump^{f}_\NN(s)\dfi^{N,n}_s
-  \lb \la \QQ,\mu^N_s\ra_E \la \genL_2 f,\mu^N_s\ra_E-\la  f,\mu^N_s\ra_E\rb,
\end{equation*}
\begin{equation*} \jump^f_\NN(s) = \frac{1}{N}\sum_{n'=1}^N \lb f\left(\pp^{N,n'},\lambda^{N,n'}_s+ \frac{\beta^C_{N,n'}}{N}\right)-f\left(\pp^{N,n'},\lambda^{N,n'}_s\right)\rb\dfi^{N,n'}_s-\frac1Nf\left(\pp^\NN,\lambda^\NN_s\right) \end{equation*}
for all $s\ge 0$, $N\in \N$ and $n\in \{1,2,\dots, N\}$,
\begin{equation*}\label{E:limgen4}
\la f,\mart^{N}_t\ra =  \frac{1}{N}\sum_{n=1}^N\int_{0}^{t}\sigma_{\NN}\sqrt{\lambda_{t}^{\NN}} \frac{\partial f}{\partial \lambda}(\hat \pp^{\NN}_s)\dfi^{N,n}_s dW^{n}_{s}+\sum_{n=1}^{N}\int_{0}^{t}\jump^f_\NN(s)d \mathcal{N}^{\NN}_{s}
\end{equation*}
and
\begin{equation*}
\mathcal{N}^{\NN}_{s}=(1-\dfi^\NN_s) - \int_{0}^s \lambda^\NN_r \dfi^\NN_r dr.
\end{equation*}

By subtracting  (\ref{Eq:LimitMu}) from (\ref{Eq:PrelimitMu}) we find that $\Xi^{N}_{t}$ satisfies the equation
\begin{align*}
\la f,\Xi^{N}_t\ra &= \la f,\Xi^{N}_0\ra+ \int_{0}^{t}\la \genL_1f+ \genL^{X_{s}}_3 f,
\Xi^{N}_s\ra ds+\int_{0}^{t}\left[ \la \QQ,\bar \mu_s\ra
\la \genL_2f,\Xi^{N}_s\ra+\la \QQ,\Xi^{N}_s\ra
\la \genL_2f, \mu^{N}_s\ra\right]ds\nonumber\\
&+\int_{0}^{t}\sqrt{N}\hat{A}^{N}[f](X_{s},\mu^{N}_{s})ds+\int_{0}^{t}\la \genL^{X_{s}}_4 f,\Xi^{N}_s\ra dV_{s}+\la f, \sqrt{N}\mart^{N}_t\ra.
\end{align*}
In order to simplify the calculations later on, for each $f\in C^{2}_{b}(\hat \PP)$, $t\ge 0$, $N\in \N$ and $n\in \{1,2,\dots, N\}$, we define the quantity
\begin{equation*}
\label{E:effcc0} \tilde \jump^f_\NN(t) \Def \frac{1}{N}\sum_{m=1}^N \beta^C_{N,m}\frac{\partial f}{\partial \lambda}(\hat \pp^{N,m}_t)\dfi^{N,m}_t-f\left(\pp^\NN,\lambda^\NN_t\right) =\la \genL_2 f,\mu^N_t\ra_E-f(\hat \pp^N_t).
\end{equation*}
A simple computation shows that
\begin{equation}
\left|\jump^f_\NN(t)-\frac1N\tilde \jump^f_\NN(t)\right|\le \frac{\KK^2}{N^2}\left\|\frac{\partial^2 f}{\partial \lambda^2}\right\|_C. \label{Eq:CoarseGrainJumpTerm}\end{equation}

We can rewrite the stochastic evolution equation that $\Xi^{N}$ satisfies as follows:
\begin{eqnarray}
\la f,\Xi^{N}_t\ra &=& \la f,\Xi^{N}_0\ra+ \int_{0}^{t}\la \genL_1f+ \genL^{X_{t}}_3 f,
\Xi^{N}_s\ra ds+\int_{0}^{t}\left[ \la \QQ,\bar \mu_s\ra
\la \genL_2f,\Xi^{N}_s\ra+\la \QQ,\Xi^{N}_s\ra
\la \genL_2f, \mu^{N}_s\ra\right]ds\nonumber\\
&+&\int_{0}^{t}\la \genL^{X_{s}}_4 f,\Xi^{N}_s\ra dV_{s}+\la f, \sqrt{N}\tilde{\mart}^{N}_t\ra+R^{N}_{t,0}\label{Eq:PrelimitEquation}
\end{eqnarray}
where
\begin{align*}
\la f, \sqrt{N}\tilde{\mart}^{N}_t\ra &=  \sqrt{N}\left(\frac{1}{N}\sum_{n=1}^N\int_{0}^{t}\sigma_{\NN}\sqrt{\lambda_{t}^{\NN}} \frac{\partial f}{\partial \lambda}(\hat \pp^{\NN}_s)\dfi^{N,n}_s dW^{n}_{s}+\frac{1}{N}\sum_{n=1}^{N}\int_{0}^{t}\tilde{\jump}^f_\NN(s)d \mathcal{N}^{\NN}_{s}\right)\\
R^{N}_{t,0}&=\int_{0}^{t}\sqrt{N}\hat{A}^{N}[f](X_{s},\mu^{N}_{s})ds+\sum_{n=1}^{N}\int_{0}^{t}\sqrt{N}\hat{B}^{\NN}[f](s)d\mathcal{N}^{\NN}_{s}\\
\hat{B}^{\NN}[f](s)&=\jump^f_\NN(s)-\frac{1}{N}\tilde{\jump}^f_\NN(s).
\end{align*}

The term $R^{N}_{t,0}$ turns out to vanish in the limit as the following lemma shows.
\begin{lemma}\label{L:AuxilliaryBounds}
For any $t\in[0,T]$ and any $f\in C^{2}_{b}(\hat{\PP})$, there is a constant $C_{0}$, independent of $n,N$ such that
\begin{equation*}
\BE\left[\int_{0}^t \sqrt{N}\left|\hat{\genA}^{N}[f](X_{r},\mu^{N}_{r})\right|dr\right]\leq \frac{C_{0}}{\sqrt{N}}t
\end{equation*}
and
\begin{equation*}
\sup_{0\leq t\leq T}\BE\left[\sum_{n=1}^{N}\int_{0}^{t}\sqrt{N}\hat{B}^{\NN}[f](s)d\mathcal{N}^{\NN}_{s}\right]^{2}\leq \frac{C_{0}}{N^{2}}T.
\end{equation*}
Moreover, we have that
\begin{equation*}
\lim_{N\rightarrow\infty}\BE\sup_{0\leq t\leq T}|R^{N}_{t,0}|^{2}=0.
\end{equation*}
Lastly, regarding the conditional (on $\mathcal{V}_{t}$) covariation of the martingale term $\la f, \sqrt{N}\tilde{\mart}^{N}_t\ra$ we have
\begin{eqnarray}
\mathrm{Cov}\left[\la f, \sqrt{N}\tilde{\mart}^{N}_{t_1}\ra, \la g, \sqrt{N}\tilde{\mart}^{N}_{t_2} \ra\Big| \mathcal{V}_{t_1 \vee t_2}\right]&=&\BE\left[\int_{0}^{t_1 \wedge t_2}\left[\la \genL_5 (f,g),\mu^{N}_s\ra+ \la \genL_6 (f,g), \mu^{N}_s\ra+
\right.\right.\nonumber\\
& &\left.\left. +\la \genL_2 f, \mu^{N}_s\ra\la \genL_2 g, \mu^{N}_s\ra\la \QQ, \mu^{N}_s\ra-\right.\right.\nonumber\\
& &\left.\left.-\la \genL_7 g, \mu^{N}_s\ra\la \genL_2 f, \mu^{N}_s\ra-\la \genL_7 f, \mu^{N}_s\ra\la \genL_2 g, \mu^{N}_s\ra\right]ds\Big|\mathcal{V}_{t_1 \vee t_2}\right].\label{L:PrelimitCovariance}
\end{eqnarray}
\end{lemma}

We continue with the proof of the theorem and defer the proof of Lemma \ref{L:AuxilliaryBounds} to the end of this section. Relative compactness of the sequences
$\left\{\Xi^{\NN}_{t}, t\in[0,T]\right\}_{N\in \N}$ and $\left\{\sqrt{N}\tilde{M}^{\NN}_{t}, t\in[0,T]\right\}_{N\in \N}$ in $D_{W^{-J}_0(w,\rho)}[0,T]$ follows by
Lemmas \ref{L:ConvergenceOfMartingale} and \ref{L:XiPrelimitSpace}. Relative compactness of the sequence
$\left\{\mu^{N}_{t}, t\in[0,T]\right\}_{N\in \N}$ in $D_{E}[0,T]$ follows by Lemma 7.1 in \citeasnoun{GieseckeSpiliopoulosSowersSirigano2012}.
These imply that the sequence
\begin{equation*}
\left\{\left(\mu^{N}, \sqrt{N}\tilde{\mart}^{N}, \Xi^{N}, V\right), N\in\N \right\}
\end{equation*}
is relatively compact in $D_{E\times W^{-J}_0(w,\rho) \times W^{-J}_0(w,\rho)\times \R}[0,T]$. Let us denote by
\begin{equation*}
\left\{\left(\bar{\mu}, \bar{\mart}, \bar{\Xi},V\right)\right\}
\end{equation*}
a limit point of this sequence.  We mention here that this is a limit in distribution, so the limit point may not be defined on the same probability space as the prelimit sequence,
but nevertheless $V$ (and thus $X$), have the same distribution both in the limit and in the prelimit. Then, by (\ref{Eq:PrelimitEquation}), Lemma \ref{L:AuxilliaryBounds} and the continuity of the operators
$\mathcal{G}$ and $\genL_{4}$ we get that $\left\{\left(\bar{\mu}, \bar{\mart}, \bar{\Xi},V\right)\right\}$ will satisfy, due to Theorem 5.5 in \citeasnoun{KurtzProtter}, the stochastic evolution equation (\ref{Eq:CLT}). By Lemma \ref{L:AuxilliaryBounds}, the conditional covariation of the distribution valued martingale $\bar{\mart}$ is given by (\ref{Eq:ConditionalCovariation}).
 Uniqueness follows by Theorem \ref{T:SEE_Uniqueness}. This concludes the proof of the theorem.

We conclude this section with the proof of Lemma \ref{L:AuxilliaryBounds}.
\begin{proof}[Proof of Lemma \ref{L:AuxilliaryBounds}]
We start by recalling that Lemma 3.4 in \citeasnoun{GieseckeSpiliopoulosSowers2011} states that for each $p\ge 1$ and $T\ge 0$,
\begin{equation}
\sup_{\substack{0\le t\le T \\ N\in \N}}\frac{1}{N}\sum_{n=1}^N\BE\left[|\lambda^\NN_t|^p\right]\leq C\label{Eq:Lemma3_4inGSS}
\end{equation}
for some constant $C>0$ that is independent of $\NN$. Along with (\ref{Eq:CoarseGrainJumpTerm}), this implies that
\begin{align}
\BE\left[\int_{0}^t \sqrt{N}\left|\hat{\genA}^{N}[f](X_{r},\mu^{N}_{r})\right|dr\right]&
=\BE\left[\int_{0}^t \sqrt{N} \left|\sum_{n=1}^N \lambda^\NN_r\jump^{f}_\NN(r)\dfi^{N,n}_r
-  \frac{1}{N}\sum_{n=1}^N \lambda^\NN_r\tilde{\jump}^{f}_\NN(r)\dfi^{N,n}_r\right|dr\right]\nonumber\\
&\leq K^2\left\|\frac{\partial^2 f}{\partial \lambda^2}\right\|_C\BE\left[\int_{0}^t \frac{1}{\sqrt{N}}\left|\frac{1}{N}\sum_{n=1}^{N}\lambda^\NN_s \right|dr\right]\nonumber\\
&\leq\frac{C_{0}}{\sqrt{N}} t\nonumber
\end{align}
for some nonnegative constant $C_{0}$ that is independent of $\NN$. Then, the first statement of the lemma follows. The second statement follows similarly. In particular, using the martingale property, (\ref{Eq:CoarseGrainJumpTerm}) and (\ref{Eq:Lemma3_4inGSS}) we immediately obtain
\begin{equation*}
\BE\left[\sum_{n=1}^{N}\int_{0}^{t}\sqrt{N}\hat{B}^{\NN}[f](s)d\mathcal{N}^{\NN}_{s}\right]^{2}
\leq \frac{C_{0}}{N^{2}} t
\end{equation*}
for some constant $C_{0}>0$. The statement regarding $R^{N}_{t,0}$ follows directly by the previous computations and Doob's inequality. The statement involving the conditional  covariation of the martingale term $\la f, \sqrt{N}\tilde{\mart}^{N}_t\ra$ follows directly by the definition of
$\tilde{\mart}^{N}$ and the fact that jumps do not occur simultaneously.
\end{proof}

\section{The appropriate weighted Sobolev space}\label{S:SobolevSpace}
In this section we discuss the Sobolev space in which the fluctuations theorem is stated. Similar weighted Sobolev spaces were introduced in \citeasnoun{Purtukhia1984} and further generalized in  \citeasnoun{GyongiKrylov1992} to study stochastic partial differential equations with unbounded coefficients. These weighted spaces turn out to be convenient and well adapted to our case. See \citeasnoun{FernandezMeleard} and \citeasnoun{KurtzXiong} for the use of weighted Sobolev spaces in fluctuation analyses of other interacting particle systems.

Let $w$ and $\rho$ be smooth functions in $\hat{\PP}$ with $w\geq0$. Let $J\geq 0$ be an integer and define by $W^{J}_{0}(w,\rho)$ the  weighted Sobolev space which is the closure of $\mathcal{C}^{\infty}_{0}(\hat{\PP})$ in the norm
\begin{equation*}
\left\Vert f\right\Vert_{W^{J}_{0}(w,\rho)}=\left(\sum_{k\leq J}\int_{\hat{\pp}\in \hat{\PP}}w^{2}(\hat{\pp})\left|\rho^{k}(\hat{\pp})D^{k} f(\hat{\pp})\right|^{2}d\hat{\pp}\right)^{1/2}<\infty
\end{equation*}

For the weight functions $w$ and $\rho$ we impose the following condition, which guarantees that $W^{J}_{0}(w,\rho)$  will be a Hilbert space, see Proposition 3.10 of \citeasnoun{GyongiKrylov1992}.

\begin{condition}\label{A:AssumptionsOnWeights}
For every $k\leq J$, the functions $\rho^{k-1}D^{k}\rho$ and $w^{-1}\rho^{k}D^{k}w$ are bounded.
\end{condition}

The inner product in this space is
\begin{equation*}
\left< f,g \right>_{W^{J}_{0}(w,\rho)}=\sum_{k\leq J}\int_{\hat{\pp}\in \hat{\PP}}w^{2}(\hat{\pp})\rho^{2k}(\hat{\pp})D^{k} f(\hat{\pp})D^{k} g(\hat{\pp})d\hat{\pp}
\end{equation*}

Moreover, we denote by $W^{-J}_{0}(w,\rho)$, the dual space to $W^{J}_{0}(w,\rho)$ equipped with the norm
\[
\left\Vert f\right\Vert_{W^{-J}_{0}(w,\rho)}=\sup_{g\in W^{J}_{0}(w,\rho)}\frac{|<f,g>|}{\left\Vert g \right\Vert_{W^{J}_{0}(w,\rho)}}
\]

For notational convenience we will sometimes write $\left\Vert f\right\Vert_{J}$ and $\left\Vert f\right\Vert_{-J}$ in the place of $\left\Vert f\right\Vert_{W^{J}_{0}(w,\rho)}$
and $\left\Vert f\right\Vert_{W^{-J}_{0}(w,\rho)}$ respectively, if no confusion arises.

The coefficients of our operators have linear growth in the first order terms and quadratic growth in the second order terms. For example, we easily see that the choices $\rho(\hat{\pp})=\sqrt{1+|\hat{\pp}|^{2}}$ and $w(\hat{\pp})=\left(1+|\hat{\pp}|^{2}\right)^{\beta}$ with $\beta<-J$ satisfy the assumptions of Condition \ref{A:AssumptionsOnWeights}. We refer the interested reader to \citeasnoun{GyongiKrylov1992} for more examples on possible choices for the weight functions $(w,\rho)$.

\begin{lemma}\label{L:DomainOfDefinitionOfG}
Let Condition \ref{A:AssumptionsOnWeights} with $J+2$ in place of $J$ hold, and fix $x\in\R$ and $\mu\in E $ such that $\int_{\hat{\PP}} \left(w^{2}(\hat{\pp})\rho^{2}(\hat{\pp})\right)^{-1}\mu(d\hat{\pp})<\infty$.
Then the operator $\mathcal{G}_{x,\mu}$ is  a linear map from $W_0^{J+2}(w,\rho)$ into $W_0^{J}(w,\rho)$ and for all $f\in W_0^{J+2}(w,\rho)$, there exists a constant $C>0$ such that
\begin{equation*}
\left\Vert \mathcal{G}_{x,\mu}f \right\Vert^{2}_{J}\leq C\left(1+|b_{0}(x)|^{2}+|\sigma_{0}(x)|^{4}+\la \QQ,\mu\ra^{2}\right)\left\Vert f\right\Vert^{2}_{J+2}
\end{equation*}
Moreover, the operator $\genL^{x}_{4}$ is  a linear map from $W_0^{J+1}(w,\rho)$ into $W_0^{J}(w,\rho)$ and for all $f\in W_0^{J+1}(w,\rho)$
\begin{equation*}
\left\Vert \genL^{x}_{4}f \right\Vert^{2}_{J}\leq C |\sigma_{0}(x)|^{2}\left\Vert f\right\Vert^{2}_{J+1}\leq C |\sigma_{0}(x)|^{2}\left\Vert f\right\Vert^{2}_{J+2}
\end{equation*}
\end{lemma}
\begin{proof}
Recall the definition of $\mathcal{G}_{x,\mu}f$ and examine term by term. For the first term, Condition \ref{A:AssumptionsOnWeights} gives us
\begin{eqnarray}
\left\Vert \genL_{1}f \right\Vert^{2}_{J}&=& \sum_{k=0}^{J}\int_{\hat{\PP}}w^{2}(\hat{\pp})\rho^{2k}(\hat{\pp})\left(D^{k}\genL_{1}f(\hat{\pp})\right)^{2}d\hat{\pp}\nonumber\\
&\leq& C \sum_{k=0}^{J+2}\int_{\hat{\PP}}w^{2}(\hat{\pp})\rho^{2k}(\hat{\pp})\left(D^{k}f(\hat{\pp})\right)^{2}d\hat{\pp}\nonumber\\
&\leq& C \left\Vert f\right\Vert^{2}_{J+2} \nonumber
\end{eqnarray}
For the second term again Condition \ref{A:AssumptionsOnWeights} gives us
\begin{eqnarray}
\left\Vert \genL_{3}^{x}f \right\Vert^{2}_{J}&=& \sum_{k=0}^{J}\int_{\hat{\PP}}w^{2}(\hat{\pp})\rho^{2k}(\hat{\pp})\left(D^{k}\genL_{3}^{x}f(\hat{\pp})\right)^{2}d\hat{\pp}\nonumber\\
&\leq&C \left(|b_{0}(x)|^{2}+|\sigma_{0}(x)|^{4}\right) \sum_{k=0}^{J+2}\int_{\hat{\PP}} w^{2}(\hat{\pp})\rho^{2k}(\hat{\pp})\left(D^{k}f(\hat{\pp})\right)^{2}d\hat{\pp}\nonumber\\
&\leq& C\left(|b_{0}(x)|^{2}+|\sigma_{0}(x)|^{4}\right)    \left\Vert f\right\Vert^{2}_{J+2} \nonumber
\end{eqnarray}
For the third term we have
\begin{eqnarray}
\left\Vert \la \QQ,\mu \ra\genL_{2}f \right\Vert^{2}_{J}&=& \la \QQ,\mu \ra^{2}\sum_{k=0}^{J}\int_{\hat{\PP}}w^{2}(\hat{\pp})\rho^{2k}(\hat{\pp})\left(D^{k}\genL_{2}f(\hat{\pp})\right)^{2}d\hat{\pp}\nonumber\\
&\leq&C \la \QQ,\mu \ra^{2}    \left\Vert f\right\Vert^{2}_{J+2} \nonumber
\end{eqnarray}
For the fourth term we have
\begin{eqnarray}
\left\Vert \la \genL_{2}f,\mu \ra \QQ \right\Vert^{2}_{J}&\leq& C \la \genL_{2}f,\mu \ra^{2}\sum_{k=0}^{J}\int_{\hat{\PP}} w^{2}(\hat{\pp})\rho^{2k}(\hat{\pp}) (1+\lambda^{2})d\hat{\pp}\nonumber\\
&\leq&C \left(\int_{\hat{\PP}}\left|D f(\hat{\pp})\right|\mu(d\hat{\pp})\right)^{2}\sum_{k=0}^{J}\int_{\hat{\PP}}w^{2}(\hat{\pp})\rho^{2k}(\hat{\pp})(1+|\hat{\pp}|^{2})d\hat{\pp}\nonumber\\
&\leq&\int_{\hat{\PP}}w^{2}(\hat{\pp})\rho^{2}(\hat{\pp})|D f(\hat{\pp})|^{2}d\hat{\pp} \int_{\hat{\PP}}\left(w^{2}(\hat{\pp})\rho^{2}(\hat{\pp})\right)^{-1}\mu(d\hat{\pp})\sum_{k=0}^{J}\int_{\hat{\PP}}w^{2}(\hat{\pp})\rho^{2k}(\hat{\pp})(1+|\hat{\pp}|^{2})d\hat{\pp}\nonumber\\
&\leq&C   \left\Vert f\right\Vert^{2}_{J} \nonumber\\
&\leq&C   \left\Vert f\right\Vert^{2}_{J+2}\label{Eq:BoundForLemmaDomainOfDefinitionOfG}
\end{eqnarray}
The statement for $\genL^{x}_{4}$ follows analogously. This completes the proof of the lemma.
\end{proof}

Notice that if the weights $w,\rho$ are chosen as above, then the condition on $\mu$ of Lemma \ref{L:DomainOfDefinitionOfG} is equivalent to assuming that $\mu\in E$ has finite moments up to order $2|\beta|-1$.

\section{Tightness and continuity properties of the limiting process}\label{S:Tightness}

We first recall three key preliminary results from \citeasnoun{GieseckeSpiliopoulosSowers2011} and \citeasnoun{GieseckeSpiliopoulosSowersSirigano2012} that will be useful in the sequel.

\begin{lemma}\label{L:macrobound}[Lemma 3.4 in \citeasnoun{GieseckeSpiliopoulosSowers2011}.] For each $p\ge 1$ and $T\ge 0$, there is a constant $C>0$, independent of $\NN$ such that
\begin{equation*} \sup_{\substack{0\le t\le T \\ N\in \N}}\frac{1}{N}\sum_{n=1}^N\BE\left[|\lambda^\NN_t|^p\right]\leq C. \end{equation*}
\end{lemma}
\begin{lemma}\label{L:bQDef}[Lemma 8.2 in \citeasnoun{GieseckeSpiliopoulosSowersSirigano2012}.] Let $W^*$ be a reference Brownian motion.  For each $\hat \pp=(\pp,\lambda_\circ)\in
\hat \PP$ where $\pp = (\alpha,\bar
\lambda,\sigma,\beta^C,\beta^S)$, there is a unique pair
$\{(Q(t),\lambda_{t}^{*}(\hat \pp)):t\in[0,T]\}$ taking values in
$\R_+\times\R_{+}$ such that
\begin{equation}\label{E:bQDef} \begin{aligned} Q(t) &= \int_{\hat \PP} \BE_{(\mathcal{V},\hat{\pp})}\left\{\lambda^{*}_{t}(\hat\pp)\exp\left[-\int_{0}^t \lambda_s^*(\hat \pp)ds\right]
\right\}\nu(d\hat{\pp}).
\end{aligned}\end{equation} and
\begin{equation}
\lambda^*_t(\hat \pp) = \lambda_\circ - \alpha\int_{0}^t (\lambda^*_s(\hat \pp)-\bar \lambda)ds + \sigma\int_{0}^t\sqrt{\lambda^*_s(\hat \pp)}dW^*_s + \beta^{C}\int_{0}^t Q(s) ds+\beta^{S}\int_{0}^t \lambda^*_s(\hat \pp)dX_{s}. \label{E:EffectiveEquation1}
\end{equation}
\end{lemma}

\begin{lemma}\label{L:Qchar}[Lemma 8.4 in \citeasnoun{GieseckeSpiliopoulosSowersSirigano2012}.] For all $A\in \Borel(\PP)$ and $B\in \Borel(\R_+)$, $\bar \mu$ satisfying (\ref{Eq:LimitMu}) is given by
\begin{equation*}\label{E:mudef} \bar \mu_t(A\times B) =  \int_{\hat \PP} \chi_A(\pp) \BE_{(\mathcal{V}_{t},\hat{\pp})}\left[\chi_B(\lambda^*_t(\hat \pp))\exp\left[-\int_{0}^t \lambda_s^*(\hat \pp)ds\right]\right]
\nu(d\hat{\pp}),
\end{equation*}
where $\lambda^*_t(\hat \pp)$ is defined via (\ref{E:bQDef})-(\ref{E:EffectiveEquation1}). Therefore, for any $f\in C(\hat \PP)$,
\begin{equation*} \la f,\bar \mu_t\ra_E = \int_{\hat \PP} \BE_{(\mathcal{V}_{t},\hat{\pp})}\left[f(\pp,\lambda_t^*(\hat \pp))\exp\left[-\int_{0}^t \lambda_s^*(\hat \pp)ds\right]\right] \nu(d\hat{\pp}). \end{equation*}
\end{lemma}

Let us now set
\begin{equation*} G(t)\Def \int_{\hat \PP}  \left(1-\BE_{(\mathcal{V},\hat{\pp})}\left[\exp\left[-\int_{0}^t \lambda_s^*(\hat \pp)ds\right]\right]\right) \nu(d\hat{\pp}).\label{Eq:DefinitionOfG}
\end{equation*}
Notice that
\begin{equation*}\label{E:gg} \dot G(t) = Q(t). \end{equation*}

We prove tightness based on a coupling argument. The role of the coupled intensities will be played by $\lambda_{t}^{\NN, G}$, defined as follows.
Define $\lambda_{t}^{\NN, G}$ to be the solution to (\ref{E:main}) when the loss process term $\beta^{C}_{\NN} L^{N}_{t}$ in the equation for the intensities has
been replaced by $G(t)$. Let us also denote by $\tau^{\NN,G}$ the corresponding default time and by $L^{N,G}_{t}$ the corresponding loss process. It is easy to see that conditional on the systematic process $X$, the $\lambda^{\NN,G}$ for $n=1,\ldots,N$ are independent. The related empirical distribution will be accordingly denoted by $\mu^{N,G}_{t}$.

Let $\kappa>0$ and define the stopping time
\begin{equation*}
\theta_{N,\kappa}^{p}=\inf \left\{ t: \frac{1}{N}\sum_{n=1}^{N}\left[|\lambda^{\NN}_{t}|^{p}+|\lambda^{\NN,G}_{t}|^{p}\right]\geq \kappa^{p} \right\}. \label{Eq:StoppingTime}
\end{equation*}
Notice that the estimate in Lemma \ref{L:macrobound} implies that for $T>0$ and $p\geq 1$
\begin{equation}
 \lim_{\kappa\rightarrow\infty}\sup_{N\in\N}\BP(\theta_{N,\kappa}^{p}\leq T)=0\label{Eq:BoundedStoppingTime}
\end{equation}
Thus, it is enough to prove tightness for $\{\Xi^{\NN}_{t\wedge \theta_{N,\kappa}}, t\in[0,T]\}_{N\in \N}$.

The following lemma provides a key estimate for the tightness proof.
\begin{lemma}\label{L:KeyEstimate}
For $J$ large enough (in particular for $J>D+1$) and weights $(w,\rho)$ such that Condition \ref{A:AssumptionsOnWeights} holds  and for every $T>0$, there is a constant $C$ independent of $N$ such that
\begin{equation*}
\sup_{N\in\N}\sup_{0\le t\le T}\BE\left(N\left\Vert \mu^N_{t\wedge \theta_{N,\kappa}}-\mu^{N,G}_{t\wedge \theta_{N,\kappa}}\right\Vert^{2}_{W^{-J}_{0}(w,\rho)}
\right)\leq C.
\end{equation*}
\end{lemma}
\begin{proof}
For notational convenience we shall write $\theta_{N}$ in place of $\theta_{N,\kappa}$. Let us denote by $\eta^{N}_{t}=\mu^{N}_{t}-\mu^{N,G}_{t}$ and let $f\in \mathcal{C}^{2}_{c}(\hat{\PP})$. For notational convenience we define the operator
\begin{equation*}
(\mathcal{G}_{x,\nu,\mu}f)(\hat \pp)=(\genL_1f)(\hat \pp)+ (\genL^{x}_3 f)(\hat \pp)+ \la \QQ,\nu\ra
 (\genL_2f)(\hat \pp)+\la \genL_2f,\mu\ra\QQ(\hat \pp)
\end{equation*}
After some term rearrangement, It\^{o}'s formula gives us, via the representations of Lemmas \ref{L:bQDef}-\ref{L:Qchar},
\begin{eqnarray}
\left<f,\eta^{N}_{t}\right>&=&\int_{0}^{t}\left<\mathcal{G}_{X_{s},\mu_{s}^{N,G},\mu_{s}^{N}}f,\eta^{N}_{s}\right>ds+\int_{0}^{t}\left<\genL_2 f, \mu^{N,G}_{s}\right>\left<\QQ,\mu^{N,G}_{s}-\bar{\mu}_{s}\right>ds+\int_{0}^{t}\left<\genL_{4}^{X_{s}}f,\eta^{N}_{s}\right>dV_{s}\nonumber\\
& &-\frac{1}{N}\sum_{n=1}^{N}\int_{0}^{t}\left<\genL_2 f, \mu^{N}_{s}\right>d\mathcal{N}^{\NN}_{s}+\sum_{n=1}^{N}\int_{0}^{t}\hat{B}^{\NN}[f](s)d\mathcal{N}^{\NN}_{s}+
\int_{0}^t \hat{\genA}^{N}[f](X_{s},\mu^{N}_{s})ds\nonumber\\
& &+\frac{1}{N}\sum_{n=1}^{N}\int_{0}^{t} f(\lambda^{\NN}_{s})d\mathcal{N}^{\NN}_{s}-\frac{1}{N}\sum_{n=1}^{N}\int_{0}^{t} f(\lambda^{\NN,G}_{s})d\mathcal{N}^{\NN,G}_{s}\nonumber\\
& &+\frac{1}{N}\sum_{n=1}^{N}\int_{0}^{t}\left(\sigma_{\NN}\sqrt{\lambda^{\NN}_{s}}\genL_2 f(\lambda^{\NN}_{s})-\sigma_{\NN}\sqrt{\lambda^{\NN,G}_{s}}\genL_2 f(\lambda^{\NN,G}_{s})\right)dW^{n}_{s}\nonumber
\end{eqnarray}

Using the bounds of Lemma \ref{L:AuxilliaryBounds}, It\^{o}'s formula and the fact that common jumps occur with probability zero, we have
\begin{eqnarray}
\BE\left<f,\eta^{N}_{t\wedge \theta_{N} }\right>^{2}&=&\int_{0}^{t}\BE\left[2\left<f,\eta^{N}_{s}\right>\left<\mathcal{G}_{X_{s},\mu_{s}^{N,G},\mu_{s}^{N}}f,\eta^{N}_{s}\right>+\left|\left<\genL_{4}^{X_{s}}f,\eta^{N}_{s}\right>\right|^{2}\right]\chi_{\{\theta_{N}\geq s\}}ds\nonumber\\
& &+ \BE \int_{0}^{t}2\left<f,\eta^{N}_{s}\right>\left<\genL_2 f, \mu^{N,G}_{s}\right>\left<\QQ,\mu^{N,G}_{s}-\bar{\mu}_{s}\right>\chi_{\{\theta_{N}\geq s\}}ds+ O\left(\frac{1}{N}\right)\nonumber
\end{eqnarray}
where the $O\left(\frac{1}{N}\right)$ term originates from the martingales and from Lemma \ref{L:AuxilliaryBounds}. Next we apply Young's inequality to get
\begin{eqnarray}
\BE\left<f,\eta^{N}_{t\wedge\theta_{N}}\right>^{2}&\leq&\int_{0}^{t}\BE\left[2\left<f,\eta^{N}_{s}\right>\left<\mathcal{G}_{X_{s},\mu_{s}^{N,G},\mu_{s}^{N}}f,\eta^{N}_{s}\right>+\left|\left<\genL_{4}^{X_{s}}f,\eta^{N}_{s}\right>\right|^{2}\right]\chi_{\{\theta_{N}\geq s\}}ds\nonumber\\
& &+ \BE \int_{0}^{t}\left[2\left<f,\eta^{N}_{s}\right>^{2}+\frac{1}{2}\left<\genL_2 f, \mu^{N,G}_{s}\right>^{2}\left<\QQ,\mu^{N,G}_{s}-\bar{\mu}_{s}\right>^{2}\right]\chi_{\{\theta_{N}\geq s\}}ds+ O\left(\frac{1}{N}\right)\nonumber\\
&=&\int_{0}^{t}\BE\left[2\left<f,\eta^{N}_{s}\right>\left<\mathcal{G}_{X_{s},\mu_{s}^{N,G},\mu_{s}^{N}}f,\eta^{N}_{s}\right>+\left|\left<\genL_{4}^{X_{s}}f,\eta^{N}_{s}\right>\right|^{2}+2\left<f,\eta^{N}_{s}\right>^{2}\right]\chi_{\{\theta_{N}\geq s\}}ds\nonumber\\
& &+ \BE \int_{0}^{t}\frac{1}{2}\left<\genL_2 f, \mu^{N,G}_{s}\right>^{2}\left<\QQ,\mu^{N,G}_{s}-\bar{\mu}_{s}\right>^{2}\chi_{\{\theta_{N}\geq s\}}ds+ O\left(\frac{1}{N}\right)\nonumber
\end{eqnarray}
Let us next bound the term on the last line of the previous equation. It is easy to see that
\begin{equation*}
Y_{s}\Def N\left<\QQ,\mu^{N,G}_{s}-\bar{\mu}_{s}\right>=\sum_{n=1}^{N}\left(\lambda^{\NN,G}_{s}\dfi^{\NN,G}_{s}-\left<\QQ,\bar{\mu}_{s}\right>\right)
\end{equation*}
is a  discrete time martingale with respect to $\BP(\cdot |\mathcal{V})$.
Thus, by Theorem 3.2 of \citeasnoun{Burkholder}, we have that
\begin{equation*}
\BE\left[\left|Y_{s}\right|^{2}\Big|\mathcal{V}\right]\leq C \BE\left[\left|S(Y_{s})\right|^{2}\Big|\mathcal{V}\right]
\end{equation*}
where, maintaining the notation of \citeasnoun{Burkholder},
\begin{equation*}
S(Y_{s})=\sqrt{\sum_{n=1}^{N}\left[\lambda^{\NN,G}_{s}\dfi^{\NN,G}_s-\la \QQ,\bar{\mu}_{s}\ra\right]^{2}}.
\end{equation*}

Therefore, recalling the bound from Lemma \ref{L:macrobound} we obtain
\begin{eqnarray}
\frac{1}{N}\BE \left|Y_{s}\right|^{2}&=&\frac{1}{N}\BE\left[\BE\left[\left|Y_{s}\right|^{2}\Big|\mathcal{V}\right]\right]\nonumber\\
&\leq&\frac{C}{N}\BE\left[\BE\left[\left|S(Y_{s})\right|^{2}\Big|\mathcal{V}\right]\right]\nonumber\\
&=&\frac{C}{N}\BE\sum_{n=1}^{N}\left[\lambda^{\NN,G}_{s}\dfi^{\NN,G}_s-\la \QQ,\bar{\mu}_{s}\ra\right]^{2}\nonumber\\
& \leq& C\nonumber
\end{eqnarray}
This bound implies that
\begin{equation*}
\BE \int_{0}^{t}\left[\frac{1}{2}\left<\genL_2 f, \mu^{N,G}_{s}\right>^{2}\left<\QQ,\mu^{N,G}_{s}-\bar{\mu}_{s}\right>^{2}\chi_{\{\theta_{N}\geq s\}}\right]ds\leq \frac{1}{N}C(t,K,\kappa)\left\Vert f\right\Vert^{2}_{W^{1}_0}
\end{equation*}

For $J$ large enough $(J>D+1)$, Proposition 3.15 of \citeasnoun{GyongiKrylov1992} and  Theorem 6.53 of \citeasnoun{Adams1978} immediately imply that the embedding $W^{J}_{0}\hookrightarrow W^{1}_{0}$ is of Hilbert-Schmidt type. So, by Lemma 1 and Theorem 2 in Chapter 2.2 of \citeasnoun{GelfandVilenkin1964}, if $\{f_{a}\}_{a\geq 1}$ is a complete orthonormal basis for $W^{J}_{0}$, then $\sum_{a\geq 1}\left\Vert f_{a}\right\Vert^{2}_{W^{1}_0}<\infty$.

Hence, if $\{f_{a}\}_{a\geq 1}$ is a complete orthonormal basis for $W^{J}_0(w,\rho)$ we obtain that
\begin{align*}
\BE\left<f_{a},\eta^{N}_{t\wedge\theta_{N}}\right>^{2}&\leq
\int_{0}^{t}\BE\left[2\left<f_{a},\eta^{N}_{s}\right>\left<\mathcal{G}_{X_{s},\mu_{s}^{N,G},\mu_{s}^{N}}f_{a},\eta^{N}_{s}\right>+\left|\left<\genL_{4}^{X_{s}}f_{a},\eta^{N}_{s}\right>\right|^{2}+2\left<f_{a},\eta^{N}_{s}\right>^{2}\right]\chi_{\{\theta_{N}\geq s\}}ds\nonumber\\
&\hspace{1cm}+
\frac{1}{N}C\left\Vert f_{a}\right\Vert^{2}_{W^{1}_0}
\end{align*}
Summing over $a\geq 1$ we then obtain
\begin{equation*}
\BE\left\Vert\eta^{N}_{t\wedge\theta_{N}}\right\Vert^{2}_{-J}\leq
\int_{0}^{t}\BE\left[2\left<\eta^{N}_{s},\mathcal{G}^{*}_{X_{s},\mu_{s}^{N,G},\mu_{s}^{N}}\eta^{N}_{s}\right>_{-J}+\left\Vert \genL_{4}^{*,X_{s}}\eta^{N}_{s}\right\Vert^{2}_{-J}+2\left\Vert\eta^{N}_{s}\right\Vert^{2}_{-J}\right]\chi_{\{\theta_{N}\geq s\}}ds+
\frac{1}{N}C.
\end{equation*}

Then as in Lemma \ref{L:Gronwall} we get that
\begin{equation*}
\BE\left[2\left<\eta^{N}_{s},\mathcal{G}^{*}_{X_{s},\mu_{s}^{N,G},\mu_{s}^{N}}\eta^{N}_{s}\right>_{-J}+\left\Vert \genL_{4}^{*,X_{s}}\eta^{N}_{s}\right\Vert^{2}_{-J}\right]\chi_{\{\theta_{N}\geq s\}}\leq C \BE\left\Vert\eta^{N}_{s}\right\Vert^{2}_{-J}\chi_{\{\theta_{N}\geq s\}}
\end{equation*}
So,
\begin{equation*}
\BE\left\Vert\eta^{N}_{t\wedge\theta_{N}}\right\Vert^{2}_{-J}\leq
C\left[\int_{0}^{t}\BE\left\Vert\eta^{N}_{s\wedge\theta_{N}}\right\Vert^{2}_{-J}ds+
\frac{1}{N}\right]
\end{equation*}
Finally, an application of Gronwall's lemma concludes the proof.
\end{proof}

The following lemma provides a uniform bound for the fluctuation process.
\begin{lemma}\label{L:KeyEstimate2}
Let $J>2D+1$ and weights $(w,\rho)$ such that Condition \ref{A:AssumptionsOnWeights} holds. For every $T>0$, there is a constant $C$ independent of $N$ such that
\begin{equation*}
\sup_{N\in\N}\sup_{0\le t\le T}\BE\left(\left\Vert \Xi^{N}_{t\wedge \theta_{N,\kappa}}\right\Vert^{2}_{W^{-J}_{0}(w,\rho)}
\right)\leq C.
\end{equation*}
\end{lemma}
\begin{proof}
Clearly $\Xi^{N}_{t}=\sqrt{N}(\mu^{N}_{t}-\mu^{N,G}_{t})+\sqrt{N}(\mu^{N,G}_{t}-\bar{\mu}_{t})=\sqrt{N}\eta^{N}_{t}+\sqrt{N}\eta^{N,G}_{t}$.
Notice now that by Lemmas \ref{L:bQDef} and \ref{L:Qchar}
\begin{equation*}
N\left<f,\eta^{N,G}_{t}\right>=\sum_{n=1}^{N}\left[f(\lambda^{\NN,G}_{t})\dfi^{\NN,G}_{t}-\left<f,\bar{\mu}_{t}\right>\right]
\end{equation*}
is a $\BP(\cdot|\mathcal{V}_{t})$ discrete time martingale, which in turn implies via Theorem 3.2 of \citeasnoun{Burkholder} and Proposition 3.15 of \citeasnoun{GyongiKrylov1992} (similarly to the proof of Lemma \ref{L:KeyEstimate}) that
\begin{equation*}
\BE\left<f,\eta^{N,G}_{t}\right>^{2}\leq \frac{1}{N}C \left\Vert f\right\Vert^{2}_{W^{D+1}_0}.
\end{equation*}

Therefore, if $\{f_{a}\}_{a\geq 1}$ is a complete orthonormal basis for $W^{J}_0(w,\rho)$ with $J>2D+1$, Parseval's identity gives (similarly to Lemma \ref{L:KeyEstimate})
\begin{equation}
 N\BE\left\Vert \eta^{N,G}_{t\wedge\theta_{N}}\right\Vert^{2}_{-J}\leq C.\label{Eq:Estimate0}
\end{equation}

Since,
\begin{equation*}
\BE\left<f_{a},\Xi^{N}_{t\wedge\theta_{N}}\right>^{2}\leq C\left[ N\BE\left<f_{a},\eta^{N}_{t\wedge\theta_{N}}\right>^{2}+N\BE\left<f_{a},\eta^{N,G}_{t\wedge\theta_{N}}\right>^{2}\right]
\end{equation*}
summing over $a\geq 1$ we then obtain
\begin{equation*}
\BE\left\Vert \Xi^{N}_{t\wedge\theta_{N}}\right\Vert^{2}_{-J}\leq C\left[N\BE\left\Vert \eta^{N}_{t\wedge\theta_{N}}\right\Vert^{2}_{-J}+N\BE\left\Vert \eta^{N,G}_{t\wedge\theta_{N}}\right\Vert^{2}_{-J}\right]
\end{equation*}
which by Lemma \ref{L:KeyEstimate} (obviously $J>2D+1>D+1$) and (\ref{Eq:Estimate0}) yields the statement of the lemma.
\end{proof}

Next we discuss relative compactness  for $\{\sqrt{N}\tilde{\mart}^{N}_{t}, t\in[0,T]\}_{N\in \N}$. In particular, we have the following lemma.
\begin{lemma}\label{L:MuPrelimitSpace}
Let $T>0$ and let $J>D+1$ and weights $(w,\rho)$ such that Condition \ref{A:AssumptionsOnWeights} holds. The process $\{\sqrt{N}\tilde{\mart}^{N}_{t}, t\in[0,T]\}_{N\in \N}$ is a $W^{-J}_0(w,\rho)-$valued martingale such that
\begin{equation*}
\sup_{N\in\N}\BE\left[\sup_{0\leq t\leq T}\left\Vert  \sqrt{N}\tilde{\mart}^{N}_{t}\right\Vert^{2}_{W^{-J}_0(w,\rho)}\right]\leq C <\infty.\label{Eq:PrelimitMart}
\end{equation*}
Furthermore, it is relatively compact in
$D_{W^{-(J+D)}_0(w,\rho)}[0,T]$.
\end{lemma}
\begin{proof}
 Clearly, it is enough to prove tightness for $\{\sqrt{N}\tilde{\mart}^{N}_{t\wedge\theta_{N,\kappa}^{p}}, t\in[0,T]\}_{N\in \N}$. Let us define
\begin{equation*}
\Gamma^{N}_{s}[f]=\la \genL_5 f, \mu^{N}_{s}\ra+\la \genL_6 f,\mu^{N}_{s}\ra+\la \genL_2 f,\mu^{N}_{s}\ra^{2}\la \QQ,\mu^{N}_{s}\ra-2\la \genL_7 f,\mu^{N}_{s}\ra\la \genL_2 f,\mu^{N}_{s}\ra
\end{equation*}

Each one of the terms in the last display can be bounded by above similarly to the derivation of the upper bound for the term $\la \genL_{2}f,\mu \ra$ in (\ref{Eq:BoundForLemmaDomainOfDefinitionOfG}). The bound from
 Lemma \ref{L:macrobound} is then used in order to treat the integrals of the weight functions with respect to $\mu^{N}_{s}$. It then follows that there exists a constant $C$ such that
\begin{equation*}
\BE\left[\int_{r}^{t}\Gamma_{s}[f]\chi_{\{\theta_{N,\kappa}^{p}\geq s\}}ds\Big |\filt^N_r\right]\leq
C(t,K,\kappa)\left\Vert f\right\Vert_{W^{1}_0}^{2}(t-r) \label{E:Gamma3}
\end{equation*}
Therefore, if $\{f_{a}\}_{a\geq 1}$ is a complete orthonormal basis for $W^{J}_0(w,\rho)$ with $J>D+1$, we obtain
\begin{eqnarray}
\BE\left[\left\Vert  \sqrt{N}\tilde{\mart}^{N}_{t\wedge \theta_{N,\kappa}^{p}}-\sqrt{N}\tilde{\mart}^{N}_{r\wedge \theta_{N,\kappa}^{p}}  \right\Vert^{2}_{W^{-J}_0(w,\rho)}|\filt^N_r\right]&\leq&
\sum_{a\geq 1}\BE\left[\la f_{a},\sqrt{N}\tilde{\mart}^{N}_{t\wedge \theta_{N,\kappa}^{p}}-\sqrt{N}\tilde{\mart}^{N}_{r\wedge \theta_{N,\kappa}^{p}}\ra^{2} |\filt^N_r\right]\nonumber\\
&\leq&C(T,K,\kappa)\sum_{a\geq 1}\BE\left[\int_{r}^{t}\Gamma_{s}[f_{a}]\chi_{\{\theta_{N,\kappa}^{p}\geq s\}}ds|\filt^N_r\right]\nonumber\\
&\leq&C(T,K,\kappa)\left[\sum_{a\geq 1}\left\Vert f_{a}\right\Vert_{W^{1}_{0}}^{2}\right] (t-r)\label{Eq:EstimateMu}
\end{eqnarray}
As in the proof of Lemma \ref{L:KeyEstimate}, the restriction $J>D+1$ implies that $\sum_{a\geq 1}\left\Vert f_{a}\right\Vert_{W^{1}_{0}}^{2}<\infty$. Similarly, we can also show that there exists a constant $C$ such that
\begin{eqnarray}
\sup_{N\in\N}\BE\left[\sup_{0\leq t\leq T}\left\Vert  \sqrt{N}\tilde{\mart}^{N}_{t\wedge \theta_{N,\kappa}^{p}} \right\Vert^{2}_{W^{-J}_0}\right]&\leq & C\nonumber
\end{eqnarray}

Moreover, due to the inequality $\left\Vert \cdot\right\Vert_{W^{-(J+D)}_0}\leq C \left\Vert \cdot\right\Vert_{W^{-J}_0}$, the inequality (\ref{Eq:EstimateMu}) implies that
\begin{eqnarray*}
\sup_{N\in\N}\BE\left[\left\Vert  \sqrt{N}\tilde{\mart}^{N}_{t\wedge \theta_{N,\kappa}^{p}}-\sqrt{N}\tilde{\mart}^{N}_{r\wedge \theta_{N,\kappa}^{p}}  \right\Vert^{2}_{W^{-(J+D)}_{0}}|\filt^N_r \right]\leq C(t-r)
\end{eqnarray*}
which obviously goes to zero as $|t-r|\downarrow 0$. The last two displays give relative compactness of $\{\sqrt{N}\tilde{\mart}^{N}_{t}, t\in[0,T]\}_{N\in \N}$ in  $D_{W^{-(J+D)}_0}[0,T]$ (Theorem 8.6 of Chapter 3 of \citeasnoun{MR88a:60130} and page 35 of \citeasnoun{JoffeMetivier1986}).

The uniform bound of the lemma follows  by the fact that $\theta_{N,\kappa}^{p}\rightarrow\infty$ as $\kappa\rightarrow\infty$, see (\ref{Eq:BoundedStoppingTime}).
\end{proof}

Regarding the convergence of the martingale $\sqrt{N}\tilde{\mart}^{N}$ we have the following lemma.
\begin{lemma}\label{L:ConvergenceOfMartingale}
Let $J>2D+1$ and weights $(w,\rho)$ such that Condition \ref{A:AssumptionsOnWeights} holds. The sequence
 $\left\{\sqrt{N}\tilde{\mart}^{N}_{t}, t\in[0,T]\right\}_{N\in \N}$ is relatively compact in $D_{W^{-J}_0(w,\rho)}[0,T]$. Moreover, it converges towards a distribution
valued martingale $\left\{\bar{\mart}_{t}, t\in[0,T]\right\}$ with conditional  (on the $\sigma-$algebra $\mathcal{V}$) covariance function, defined, for $f,g\in W^{J}_0(w,\rho)$, by (\ref{Eq:ConditionalCovariation}). The martingale $\left\{\bar{\mart}_{t}, t\in[0,T]\right\}$ is conditionally on the $\sigma-$algebra $\mathcal{V}$, Gaussian.
\end{lemma}
\begin{proof}
 Relative compactness follows by Lemma \ref{L:MuPrelimitSpace}. Due to continuous dependence of (\ref{L:PrelimitCovariance}) on $\mu^{N}$ and on the weak convergence of
$\mu^{N}_{\cdot}\rightarrow\bar{\mu}_{\cdot}$ by \citeasnoun{GieseckeSpiliopoulosSowersSirigano2012}, we obtain that any limit point of $\sqrt{N}\tilde{\mart}^{N}_{t}$ as $N\rightarrow\infty$, $\bar{\mart}$, will satisfy  (\ref{Eq:ConditionalCovariation}).
Conditionally on $\mathcal{V}$, the limiting $\bar{\mart}$ is a continuous square integrable martingale and its  predictable variation  is deterministic. Thus, by Theorem 7.1.4 in \citeasnoun{MR88a:60130}, it is conditionally Gaussian.
This concludes the proof.
\end{proof}

Next we discuss relative compactness of the process $\left\{\Xi^{\NN}_{t}, t\in[0,T]\right\}_{N\in \N}$.

\begin{lemma}\label{L:XiPrelimitSpace}
Let $T>0$, $J>3D+1$ and weights $(w,\rho)$ such that Condition \ref{A:AssumptionsOnWeights} holds. The process $\{\Xi^{\NN}_{t}, t\in[0,T]\}_{N\in \N}$ is relatively compact in  $D_{W^{-J}_0(w,\rho)}[0,T]$.

\end{lemma}
\begin{proof}
It is enough to prove tightness for
$\{\Xi^{\NN}_{t\wedge \theta_{N,\kappa}}, t\in[0,T]\}_{N\in \N}$. For $J_{1}>2D+1$, the bound from Lemma \ref{L:KeyEstimate2} holds, i.e.,
\begin{equation}
\sup_{N\in\N}\sup_{0\le t\le T}\BE\left(\left\Vert \Xi^{N}_{t\wedge \theta_{N,\kappa}}\right\Vert^{2}_{W^{-J_{1}}_{0}}
\right)\leq C.\label{Eq:BoundWithJ1}
\end{equation}

Let $\{f_{a}\}_{a\geq 1}$ be a complete orthonormal basis for $W^{J}_0$ with $J>J_{1}+D>3D+1$. By (\ref{Eq:PrelimitEquation}) we have
\begin{equation}
\la f_{a},\Xi^{N}_t\ra = \la f_{a},\Xi^{N}_r\ra+ \int_{r}^{t}\la \mathcal{G}_{X_{s},\bar{\mu}_{s},\mu^{N}_{s}}f_{a},
\Xi^{N}_s\ra ds+\int_{r}^{t}\la \genL^{X_{s}}_4 f_{a},\Xi^{N}_s\ra dV_{s}+\la f_{a}, \sqrt{N}\tilde{\mart}^{N}_{t,r}\ra+R^{N}_{t,r}.\label{Eq:representationXi}
\end{equation}

Next, we consider the mapping $H$ from $W^{J}_0(w,\rho)$ into $\mathbb{R}$ defined by
\[
H(f)=\left<\mathcal{G}_{X_{s},\bar{\mu}_{s},\mu^{N}_{s}}f,\Xi^{N}_s\right>
\]
and we notice that
\begin{align}
 \left<\mathcal{G}_{X_{s},\bar{\mu}_{s},\mu^{N}_{s}}f,\Xi^{N}_s\right>
&\leq \left\Vert \mathcal{G}_{X_{s},\bar{\mu}_{s},\mu^{N}_{s}}f\right\Vert^{2}_{W^{J_{1}}_0}\left\Vert \Xi^{N}_{s}\right\Vert^{2}_{W^{-J_{1}}_0}\nonumber\\
&\leq C \left\Vert f\right\Vert^{2}_{W^{J_{1}+2}_0}\left\Vert \Xi^{N}_{s}\right\Vert^{2}_{W^{-J_{1}}_0}\nonumber\\
&\leq C \left\Vert f\right\Vert^{2}_{W^{J_{1}+D}_0}\left\Vert \Xi^{N}_{s}\right\Vert^{2}_{W^{-J_{1}}_0}\nonumber
\end{align}
where the second inequality is due to Lemma \ref{L:DomainOfDefinitionOfG} and the third inequality because $D>2$.
Hence, by Parseval's identity we have
\begin{align*}
\left\Vert H\right\Vert^{2}_{W^{-J}_0}&\leq C \left\Vert \Xi^{N}_{s}\right\Vert^{2}_{W^{-J_{1}}_0}
\end{align*}

Thus, by (\ref{Eq:BoundWithJ1}) we get
\begin{align*}
 \sum_{a\geq 1}\BE\left[\int_{r}^{t}\left<\mathcal{G}_{X_{s},\bar{\mu}_{s},\mu^{N}_{s}}f_{a},\Xi^{N}_s\right>^{2}\chi_{\{\theta_{N,\kappa}^{p}\geq s\}}ds|\filt^{N}_{r}\right]
&\leq C\BE\left[\int_{r}^{t}\left\Vert \Xi^{N}_{s}\right\Vert^{2}_{W^{-J_{1}}_0}\chi_{\{\theta_{N,\kappa}^{p}\geq s\}}ds|\filt^{N}_{r}\right]\nonumber\\
&\leq C(t-r)
\end{align*}

Similarly we also obtain that
\begin{equation*}
\sum_{a\geq 1}\BE\left[\int_{r}^{t}\left<\genL^{X_{s}}_{4}f_{a},\Xi^{N}_s\right>^{2}\chi_{\{\theta_{N,\kappa}^{p}\geq s\}}ds|\filt^{N}_{r}\right]\leq C (t-r)
\end{equation*}
The last estimates, the uniform bound from Lemma \ref{L:KeyEstimate2}, Lemma \ref{L:MuPrelimitSpace} for $\sqrt{N}\tilde{\mart}^{N}_{t}$ and Lemma \ref{L:AuxilliaryBounds} for the remainder term $R^{N}_{t,r}$ imply the statement of the lemma. We follow the same steps as in the proof of Lemma \ref{L:MuPrelimitSpace} and thus the details are omitted.
\end{proof}

We end this section by proving a continuity result. 

\begin{lemma}\label{L:Continuity}
Any limit point of $\{\Xi^{N}_{t}, t\in[0,T]\}_{N\in \N}$ and $\{\sqrt{N}\tilde{N}^{N}_{t}, t\in[0,T]\}_{N\in \N}$ is continuous, i.e.,  it takes values in $C_{W^{-J}_0(w,\rho)}[0,T]$, with $J>3D+1$.
\end{lemma}
\begin{proof}
In order to prove that any limit point of $\left\{\Xi^{N}_{t}, t\in[0,T]\right\}_{N\in \N}$ takes values in $C_{W^{-J}_0(w,\rho)}[0,T]$, it is sufficient to show that
\[
\lim_{N\rightarrow\infty}\BE\left[\sup_{t\leq T}\left\Vert \Xi^{N}_{t}-\Xi^{N}_{t-} \right\Vert_{W^{-J}_0(w,\rho)}\right]=0
\]
Let $\{f_{a}\}_{a\geq 1}$ be a complete orthonormal basis for $W^{J}_0(w,\rho)$. Then, by definition, we have
\begin{align*}
\left<f_{a},\Xi^{N}_{t}-\Xi^{N}_{t-}  \right>&=\sqrt{N}\left[\left<f_{a},\mu^{N}_{t}-\bar{\mu}_{t}\right>-\left<f_{a},\mu^{N}_{t-}-\bar{\mu}_{t-}\right>\right]\nonumber\\
&=\sqrt{N}\left[\left<f_{a},\mu^{N}_{t}-\mu^{N}_{t-}\right>-\left<f_{a},\bar{\mu}_{t}-\bar{\mu}_{t-}\right>\right]\nonumber\\
&=\sqrt{N}\left[\jump^{f_{a}}_{\NN}(t)-\left<f_{a},\bar{\mu}_{t}-\bar{\mu}_{t-}\right>\right]\nonumber\\
&=\sqrt{N}\left[\jump^{f_{a}}_{\NN}(t)-\frac{1}{N}\tilde{\jump}^{f_{a}}_{\NN}(t)\right]+\frac{1}{\sqrt{N}}\tilde{\jump}^{f_{a}}_{\NN}(t)
-\sqrt{N}\left[\left<f_{a},\bar{\mu}_{t}-\bar{\mu}_{t-}\right>\right].\nonumber
\end{align*}
Clearly, $t\mapsto\bar{\mu}_{t}$ is continuous, so we only need to consider the first two terms. The first term is bounded by $\frac{\KK^2}{N^2}\left\|\frac{\partial^2 f_{a}}{\partial \lambda^2}\right\|_C$, see (\ref{Eq:CoarseGrainJumpTerm}), whereas for the second we immediately get
\[
\left|\tilde{\jump}^{f_{a}}_{\NN}(t)\right|\leq \KK \left[\left\Vert f_{a}\right\Vert+\left\Vert f'_{a}\right\Vert\right].
\]
So, as in Proposition 3.15 of \citeasnoun{GyongiKrylov1992}, we  get  that
\begin{equation*}
\BE\sup_{t\in[0,T]}\left|\left<f_{a},\Xi^{N}_{t}-\Xi^{N}_{t-}  \right>\right|\leq C_{0}\left[\frac{1}{N^{3/2}}+\frac{1}{\sqrt{N}}\right]\left\Vert f_{a}\right\Vert_{W^{D+3}_0}.
\end{equation*}
Thus,
\begin{align*}
\BE\left[\sup_{t\leq T}\left\Vert \Xi^{N}_{t}-\Xi^{N}_{t-} \right\Vert_{W^{-J}_0(w,\rho)}\right]&\leq \sum_{a\geq 1}\BE\left[\sup_{t\leq T}\left<f_{a},\Xi^{N}_{t}-\Xi^{N}_{t-} \right>\right]\nonumber\\
&\leq C_{0}\left[\frac{1}{N^{3/2}}+\frac{1}{\sqrt{N}}\right] \sum_{a\geq 1}\left\Vert f_{a}\right\Vert_{W^{D+3}_0}
\end{align*}

Since $J>3D+1$, we certainly have $W^{J}_{0}\hookrightarrow W^{D+3}_{0}$, and then as in Lemma \ref{L:KeyEstimate}, $\sum_{a\geq 1}\left\Vert f_{a}\right\Vert_{W^{D+3}_0}<\infty$, which implies that after taking the limit as $N\rightarrow\infty$, the right hand side of the last display goes to zero. This concludes the proof of continuity of the limit point trajectories of $\{\Xi^{N}_{t}, t\in[0,T]\}_{N\in \N}$.

Next we consider continuity of the trajectories of the limit points of $\{\sqrt{N}\tilde{N}^{N}_{t}, t\in[0,T]\}_{N\in \N}$. It follows directly by (\ref{Eq:representationXi}) and Lemma \ref{L:AuxilliaryBounds} that $\Xi^{N}_{t}$ and $\sqrt{N}\tilde{M}^{N}_{t}$ have the same discontinuities. Thus, the continuity of any limit point of $\{\Xi^{N}_{t}, t\in[0,T]\}_{N\in \N}$ implies the continuity of any limit point of $\{\sqrt{N}\tilde{N}^{N}_{t}, t\in[0,T]\}_{N\in \N}$, which concludes the proof of the lemma.
\end{proof}

\section{Uniqueness}\label{S:Uniqueness}

In this section we prove uniqueness of the stochastic evolution equation (\ref{Eq:CLT}). Let $\bar{\Xi}^{1}_{t}, \bar{\Xi}^{1}_{t}$ be two solutions of (\ref{Eq:CLT}) and let us define $\Phi_{t}=\bar{\Xi}^{1}_{t}-\bar{\Xi}^{2}_{t}$. Then, $\Phi_{t}$ will satisfy the stochastic evolution equation

\begin{align*}
\la f,\Phi_t\ra &=  \int_{0}^{t}\la \genL_1f+ \genL^{X_{s}}_3 f,
\Phi_s\ra ds+\int_{0}^{t}\left[ \la \QQ,\bar \mu_s\ra
\la \genL_2f,\Phi_s\ra+\la \QQ,\Phi_s\ra
\la \genL_2f,\bar \mu_s\ra\right]ds+\int_{0}^{t}\la \genL^{X_{s}}_4 f,\Phi_s\ra dV_{s}\nonumber\\
&= \int_{0}^{t}\la \mathcal{G}_{X_{s},\bar{\mu}_{s}}f,\Phi_s\ra ds+\int_{0}^{t}\la \genL^{X_{s}}_4 f,\Phi_s\ra dV_{s} \quad a.s. \label{Eq:UniquenessCLT1}
\end{align*}
Notice that this is a linear equation. In order to prove uniqueness, it is enough
to show that

\begin{equation*}
\BE \left\Vert \Phi_t\right\Vert^{2}_{-J}=0
\end{equation*}

Let $\{f_{a}\}$ be a complete orthonormal basis for $W^{J}_0(w,\rho)$. By It\^{o}'s formula we get that
\begin{align*}
 \left|\la f_{a},\Phi_t\ra\right|^2 &=  \int_{0}^{t}2\la f_{a},\Phi_s\ra\left[\la \genL_1f_{a}+ \genL^{X_{s}}_3 f_{a},
\Phi_s\ra+\la \QQ,\bar \mu_s\ra
\la \genL_2f_{a},\Phi_s\ra+\la \QQ,\Phi_s\ra
\la \genL_2f_{a},\bar \mu_s\ra\right]ds+\nonumber\\
&\hspace{2cm}+\int_{0}^{t}\left|\la \genL^{X_{s}}_4 f_{a},\Phi_s\ra\right|^{2} ds+\int_{0}^{t}2\la f_{a},\Phi_s\ra \la \genL^{X_{s}}_4 f_{a},\Phi_s\ra dV_{s}\nonumber\\
&=\int_{0}^{t}2\la f_{a},\Phi_s\ra\la \mathcal{G}_{X_{s},\bar{\mu}_{s}}f_{a},
\Phi_s\ra ds+\int_{0}^{t}\left|\la \genL^{X_{s}}_4 f_{a},\Phi_s\ra\right|^{2} ds+\int_{0}^{t}2\la f_{a},\Phi_s\ra \la \genL^{X_{s}}_4 f_{a},\Phi_s\ra dV_{s}\label{Eq:UniquenessCLT2}
\end{align*}
Then, summing over $a$ and taking expected value (due to Lemma \ref{L:DomainOfDefinitionOfG}, the expected value of the stochastic integral is zero) we have
\begin{equation}
\BE \left\Vert \Phi_t\right\Vert^{2}_{-J} =  \BE\int_{0}^{t}2\la\Phi_s,\mathcal{G}^{*}_{s}\Phi_s\ra_{-J} ds+\BE\int_{0}^{t}\left\Vert \genL^{*,X_{s}}_4 \Phi_s\right\Vert^{2}_{-J}  ds\label{Eq:UniquenessCLT3}
\end{equation}

Hence, if we prove that there is a constant $C>0$ such that
\begin{equation*}
2\la \phi, \mathcal{G}^{*}\phi\ra_{-J}+\left\Vert \genL^{*,x}_4 \phi\right\Vert^{2}_{-J} \leq  C \left\Vert \phi\right\Vert^{2}_{-J}
\end{equation*}
then we can conclude by Gronwall inequality that $\Phi_{t}=0$ a.s.

Let us recall now that
\begin{equation*}
(\mathcal{G}_{x,\mu}f)(\hat \pp)=(\genL_1f)(\hat \pp)+ (\genL^{x}_3 f)(\hat \pp)+ \la \QQ, \mu\ra
 (\genL_2f)(\hat \pp)+\la \genL_2f, \mu\ra\QQ(\hat \pp)
\end{equation*}

and set
\begin{eqnarray}
(\mathcal{G}^{1}_{x,\mu}f)(\hat \pp)&=&(\genL_1f)(\hat \pp)+ (\genL^{x}_3 f)(\hat \pp)+ \la \QQ,\mu\ra
 (\genL_2f)(\hat \pp)\nonumber\\
 (\mathcal{G}^{2}_{x,\mu}f)(\hat \pp)&=&\la \genL_2f, \mu\ra\QQ(\hat \pp)\nonumber
\end{eqnarray}
and $(\mathcal{G}^{i}_{s}f)(\hat \pp)=(\mathcal{G}^{i}_{X_{s},\bar{\mu}_{s}}f)(\hat \pp)$ for i=1,2.

Moreover, for simplicity in presentation and without loss of generality, we shall consider from now on only the case of a  homogeneous portfolio, i.e, set $\nu=\delta_{\hat \pp_{0}}$. This is done without loss of generality, since the operators $\genL_if$ involve differentiation only with respect to $\lambda$. The proof for the general heterogeneous case, is identical with heavier notation.

Before a term by term examination, we gather some straightforward results in the following lemma.
\begin{lemma}\label{L:IntegrationByParts}
Let $\psi\in C^{\infty}_{c}(\R_{+})$ and $J\geq 1$ and assume Condition \ref{A:AssumptionsOnWeights}. Then, up to a multiplicative constant for the term $O\left( \left\Vert\psi\right\Vert^{2}_{J}\right)$, that may be different from line to line,  we have
\begin{eqnarray}
\sum_{k=0}^{J}\int_{\R_{+}}w^{2}(\lambda)\rho^{2k}(\lambda)\lambda\psi^{(k+1)}(\lambda)\psi^{(k)}(\lambda)d\lambda&= &O\left( \left\Vert\psi\right\Vert^{2}_{J}\right)\nonumber\\
\sum_{k=0}^{J}\int_{\R_{+}}w^{2}(\lambda)\rho^{2k}(\lambda)\psi^{(k+1)}(\lambda)\psi^{(k)}(\lambda)d\lambda&=& O\left( \left\Vert\psi\right\Vert^{2}_{J}\right)\nonumber\\
\sum_{k=0}^{J}\int_{\R_{+}}w^{2}(\lambda)\rho^{2k}(\lambda)\lambda\psi^{(k+2)}(\lambda)\psi^{(k)}(\lambda)d\lambda&=&
-\int_{\R_{+}}w^{2}(\lambda)\rho^{2k}(\lambda)\lambda\left|\psi^{(k+1)}(\lambda)\right|^{2}d\lambda+O\left( \left\Vert\psi\right\Vert^{2}_{J}\right)\nonumber\\
\sum_{k=0}^{J}\int_{\R_{+}}w^{2}(\lambda)\rho^{2k}(\lambda)\lambda^{2}\psi^{(k+2)}(\lambda)\psi^{(k)}(\lambda)d\lambda&=&
-\int_{\R_{+}}w^{2}(\lambda)\rho^{2k}(\lambda)\left|\lambda \psi^{(k+1)}(\lambda)\right|^{2}d\lambda
+O\left( \left\Vert\psi\right\Vert^{2}_{J}\right)\nonumber
\end{eqnarray}
\end{lemma}
\begin{proof}
It follows directly by integration by parts using Condition \ref{A:AssumptionsOnWeights} and the assumption that $\psi$ and its derivatives are compactly supported.
\end{proof}

Then we bound each term on the right hand side of (\ref{Eq:UniquenessCLT3}).
First, we notice that, by Riesz representation theorem, for $\phi\in W^{-J}_{0}$ there exists a unique $\psi=F(\phi)\in W^{J}_{0}$ such that
\begin{equation*}
\la \phi, f\ra=\la\psi, f\ra_{J}
\end{equation*}
By a density argument we may assume that $\psi=F(\phi)\in W^{J+2}_{0}$  and obtain
\begin{equation*}
\la \phi,\mathcal{G}^{*}\phi\ra_{-J} = \la \psi,\mathcal{G}^{*}\phi\ra =\la \mathcal{G} \psi,\phi\ra =\la \mathcal{G} \psi,\psi\ra_{J}.
\end{equation*}
which is true since by Lemma \ref{L:DomainOfDefinitionOfG}, $\mathcal{G} \psi\in W^{J}_{0}$.

\begin{lemma}\label{L:Bounds2a}
For $\phi\in W^{-J}_{0}$ such that $\psi=F(\phi)\in W^{J+2}_{0}$ we have
\begin{eqnarray*}
\la \phi, \mathcal{G}^{1,*}\phi\ra_{-J}&=&  \left(1+|b_{0}(x)|+|\sigma_{0}(x)|^{2}+\left<\QQ,\mu\right>\right)O\left(\left\Vert \phi\right\Vert^{2}_{-J}\right)-\frac{1}{2}\left|\beta^{S}\sigma_{0}(x)\right|^{2}\sum_{k=0}^{J}\int_{\R_{+}}w^{2}(\lambda)\rho^{2k}(\lambda) \left|\lambda\psi^{(k+1)}(\lambda)\right|^{2}d\lambda\nonumber\\
& &-\frac{1}{2}\sigma^{2}\sum_{k=0}^{J}\int_{\R_{+}}w^{2}(\lambda)\rho^{2k}(\lambda)\lambda \left|\psi^{(k+1)}(\lambda)\right|^{2}d\lambda-\sum_{k=0}^{J}\int_{\R_{+}}w^{2}(\lambda)\rho^{2k}(\lambda)\lambda \left|\psi^{(k)}(\lambda)\right|^{2}d\lambda
\end{eqnarray*}Since $\psi\in W^{J+2}_{0}$, the integrals on the right hand side are well defined and bounded.
\end{lemma}
\begin{proof}
We recall that
\begin{equation*}
\la \phi,\mathcal{G}^{1,*}\phi\ra_{-J} = \la \mathcal{G}^{1} \psi,\psi\ra_{J}.
\end{equation*}
and that
\begin{equation*}
(\mathcal{G}^{1}\psi)(\hat \pp)=(\genL_1\psi)(\hat \pp)+ (\genL^{x}_3 \psi)(\hat \pp) + \la \QQ, \mu \ra
 (\genL_2\psi)(\hat \pp)
\end{equation*}
Hence we bound each of the three terms separately. Using the statements of Lemma \ref{L:IntegrationByParts} we have for $\genL_1\psi$

\begin{align}
& \sum_{k=0}^{J}\int_{\R_{+}}w^{2}(\lambda)\rho^{2k}(\lambda)\psi^{(k)}(\lambda)\left(\genL_1 \psi(\lambda)\right)^{(k)}d\lambda=\nonumber\\
&\hspace{1cm}=\frac{1}{2}\sigma^{2}\sum_{k=0}^{J}\int_{\R_{+}} w^{2}(\lambda)\rho^{2k}(\lambda) \psi^{(k)}(\lambda)\left(\lambda \psi^{(2)}(\lambda)\right)^{(k)}d\lambda-\alpha\sum_{k=0}^{J}\int_{\R_{+}}w^{2}(\lambda)\rho^{2k}(\lambda)\psi^{(k)}(\lambda)\left(\lambda \psi^{(1)}(\lambda)\right)^{(k)}d\lambda\nonumber\\
&\hspace{1cm}+\alpha\bar{\lambda}\sum_{k=0}^{J}\int_{\R_{+}}w^{2}(\lambda)\rho^{2k}(\lambda)\psi^{(k)}(\lambda) \psi^{(k+1)}(\lambda)d\lambda-\sum_{k=0}^{J}\int_{\R_{+}}w^{2}(\lambda)\rho^{2k}(\lambda)\psi^{(k)}(\lambda)\left(\lambda \psi(\lambda)\right)^{(k)}d\lambda\nonumber\\
&\hspace{1cm}=\frac{1}{2}\sigma^{2}\sum_{k=0}^{J}\int_{\R_{+}}w^{2}(\lambda)\rho^{2k}(\lambda)\psi^{(k)}(\lambda)\left(\kappa\psi^{(k+1)}(\lambda)+\lambda \psi^{(k+2)}(\lambda)\right)d\lambda\nonumber\\
&\hspace{1cm}-\sum_{k=0}^{J}\alpha\int_{\R_{+}}w^{2}(\lambda)\rho^{2k}(\lambda)\psi^{(k)}(\lambda)\left(\kappa\psi^{(k)}(\lambda)+\lambda \psi^{(k+1)}(\lambda)\right)d\lambda\nonumber\\
& \hspace{1cm}+\sum_{k=0}^{J}\alpha\bar{\lambda}\int_{\R_{+}}w^{2}(\lambda)\rho^{2k}(\lambda)\psi^{(k)}(\lambda) \psi^{(k+1)}(\lambda)d\lambda-\sum_{k=0}^{J}\int_{\R_{+}}w^{2}(\lambda)\rho^{2k}(\lambda)\psi^{(k)}(\lambda)\left(\kappa\psi^{(k-1)}(\lambda)+\lambda \psi^{(k)}(\lambda)\right)d\lambda\nonumber\\
&\hspace{1cm}= O\left(\left\Vert\psi\right\Vert^{2}_{J}\right)-\frac{1}{2}\sigma^{2}\sum_{k=0}^{J}\int_{\R_{+}}w^{2}(\lambda)\rho^{2k}(\lambda)\lambda \left|\psi^{(k+1)}(\lambda)\right|^{2}d\lambda-\sum_{k=0}^{J}\int_{\R_{+}}w^{2}(\lambda)\rho^{2k}(\lambda)\lambda \left|\psi^{(k)}(\lambda)\right|^{2}d\lambda\nonumber\\
&\hspace{1cm}= O\left(\left\Vert\phi\right\Vert^{2}_{-J}\right)-\frac{1}{2}\sigma^{2}\sum_{k=0}^{J}\int_{\R_{+}}w^{2}(\lambda)\rho^{2k}(\lambda)\lambda \left|\psi^{(k+1)}(\lambda)\right|^{2}d\lambda-\sum_{k=0}^{J}\int_{\R_{+}}w^{2}(\lambda)\rho^{2k}(\lambda)\lambda \left|\psi^{(k)}(\lambda)\right|^{2}d\lambda\nonumber
\end{align}

Similarly, for  $\genL_{3}^{x}\psi$ we have
\begin{align}
& \sum_{k=0}^{J}\int_{\R_{+}}w^{2}(\lambda)\rho^{2k}(\lambda)\psi^{(k)}(\lambda)\left(\genL_{3}^{x} \psi
(\lambda)\right)^{(k)}d\lambda=\nonumber\\
&\hspace{1cm}=\beta^{S}b_{0}(x)\sum_{k=0}^{J}\int_{\R_{+}}w^{2}(\lambda)\rho^{2k}(\lambda)\psi^{(k)}(\lambda)\left(\lambda \psi^{(1)}(\lambda)\right)^{(k)}d\lambda\nonumber\\
&\hspace{1cm}+\frac{1}{2}\left|\beta^{S}\sigma_{0}(x)\right|^{2}\sum_{k=0}^{J}\int_{\R_{+}}w^{2}(\lambda)\rho^{2k}(\lambda)\psi^{(k)}(\lambda)\left(\lambda^{2} \psi^{(2)}(\lambda)\right)^{(k)}d\lambda\nonumber\\
&\hspace{1cm}=\beta^{S}b_{0}(x)\sum_{k=0}^{J}\int_{\R_{+}}w^{2}(\lambda)\rho^{2k}(\lambda)\psi^{(k)}(\lambda)\left(k\psi^{(k)}(\lambda)+\lambda \psi^{(k+1)}(\lambda)\right)d\lambda\nonumber\\
& \hspace{1cm}+\frac{1}{2}\left|\beta^{S}\sigma_{0}(x)\right|^{2}\sum_{k=0}^{J}\int_{\R_{+}}w^{2}(\lambda)\rho^{2k}(\lambda)\psi^{(k)}(\lambda)\left(\lambda^{2} \psi^{(k+2)}(\lambda)+2k\lambda \psi^{(k+1)}+w^{2}(\lambda)\rho^{2k}(\lambda)\frac{k!}{(k-2)!2!}\psi^{(k)}\right)d\lambda\nonumber\\
&\hspace{1cm}=\left(b_{0}(x)+\left|\beta^{S}\sigma_{0}(x)\right|^{2}\right) O\left(\left\Vert\psi\right\Vert^{2}_{J}\right)
-\frac{1}{2}\left|\beta^{S}\sigma_{0}(x)\right|^{2}\sum_{k=0}^{J}\int_{\R_{+}}w^{2}(\lambda)\rho^{2k}(\lambda)\left|\lambda\psi^{(k+1)}(\lambda)\right|^{2}d\lambda\nonumber\\
&\hspace{1cm}=\left(b_{0}(x)+\left|\beta^{S}\sigma_{0}(x)\right|^{2}\right) O\left(\left\Vert\phi\right\Vert^{2}_{-J}\right)
-\frac{1}{2}\left|\beta^{S}\sigma_{0}(x)\right|^{2}\sum_{k=0}^{J}\int_{\R_{+}}w^{2}(\lambda)\rho^{2k}(\lambda)\left|\lambda\psi^{(k+1)}(\lambda)\right|^{2}d\lambda\nonumber
\end{align}

Lastly, for the third term, $ \la \QQ, \mu\ra (\genL_2\psi)(\hat \pp)$, we have

\begin{eqnarray}
\la \QQ, \mu\ra\sum_{k=0}^{J}\int_{\R_{+}}w^{2}(\lambda)\rho^{2k}(\lambda)\psi^{(k)}(\lambda)\left(\genL_{2} \psi(\lambda)\right)^{(k)}d\lambda&=&\la \QQ,\mu\ra\sum_{k=0}^{J}\int_{\R_{+}}w^{2}(\lambda)\rho^{2k}(\lambda)\psi^{(k)}(\lambda)\psi^{(k+1)}(\lambda)d\lambda\nonumber\\
&=&\la \QQ,\mu\ra O\left(\left\Vert\phi\right\Vert^{2}_{-J}\right)\nonumber
\end{eqnarray}

Hence, putting all the estimates together, we conclude the proof of the lemma.
\end{proof}

\begin{lemma}\label{L:Bounds3a}
Let us assume that $\mu$ is such that $\int_{\R_{+}} \left(w^{2}(\lambda)\rho^{2}(\lambda)\right)^{-1}\mu(d\lambda)<\infty$. For $\phi\in W^{-J}_{0}$ such that $\psi\in W^{J}_{0}$ we have
\begin{equation*}
\la \phi, \mathcal{G}^{2,*}\phi\ra_{-J}=O\left(\left\Vert\phi\right\Vert^{2}_{-J}\right)
\end{equation*}
\end{lemma}
\begin{proof}
By definition we have

\begin{eqnarray}
\la \mathcal{G}^{2}\psi, \psi\ra_{J}&=& \la \genL_{2}\psi,\mu \ra\sum_{k=0}^{J}\int_{\R_{+}}w^{2}(\lambda)\rho^{2k}(\lambda)\lambda^{(k)}\psi^{(k)}(\lambda)d\lambda\nonumber\\
&\leq&  \sqrt{\left\Vert\psi\right\Vert_{1}\int_{\R_{+}} \left(w^{2}(\lambda)\rho^{2}(\lambda)\right)^{-1}\mu(d\lambda)}\times\nonumber\\
& &\hspace{1cm}\times\left[\left(\int_{\R_{+}}w^{2}(\lambda)\lambda^{2}d\lambda\right)^{1/2}\left\Vert\psi\right\Vert_{0}+
\left(\int_{\R_{+}}w^{2}(\lambda)\rho^{2}(\lambda)d\lambda\right)^{1/2}\left\Vert\psi\right\Vert_{1}\right]\nonumber\\
&\leq& C \left\Vert\psi\right\Vert^{2}_{J}\nonumber\\
&=& C \left\Vert\phi\right\Vert^{2}_{-J}\nonumber
\end{eqnarray}
which concludes the proof of the lemma.
\end{proof}

\begin{remark}
Notice that the condition on the integrability of $\mu$ that appears in the statement of Lemma \ref{L:Bounds3a}, $\int_{\R_{+}} \left(w^{2}(\lambda)\rho^{2}(\lambda)\right)^{-1}\mu(d\lambda)<\infty$, is equivalent to assuming that $\mu$ has finite moments at least up to order
$2|\beta|-1$ if the weights are chosen such that $\rho(\lambda)=\sqrt{1+|\lambda|^{2}}$ and $w(\lambda)=\left(1+|\lambda|^{2}\right)^{\beta}$ and $\beta<0$.
\end{remark}

\begin{lemma}\label{L:Bounds1}
For $\phi\in W^{-J}_{0}$ such that $\psi\in W^{J+1}_{0}$ we have
\begin{equation*}
\left\Vert \genL^{*,x}_4 \phi\right\Vert^{2}_{-J}= \left|\beta^{S}\sigma_{0}(x)\right|^{2}O\left(\left\Vert \phi\right\Vert^{2}_{-J}\right)+\left|\beta^{S}\sigma_{0}(x)\right|^{2}
\sum_{k=0}^{J}\int_{\R_{+}}w^{2}(\lambda)\rho^{2k}(\lambda)\left|\lambda\psi^{(k+1)}(\lambda)\right|^{2}d\lambda
\end{equation*}
Since $\psi\in W^{J+1}_{0}$ the integral on the right hand side is well defined and bounded.
\end{lemma}
\begin{proof}
By definition we have

\begin{equation*}
\left\Vert \genL^{*,x}_4 \phi\right\Vert_{-J}=\sup_{\zeta\in W^{J}}\frac{\left|\la\genL^{*,x}_4 \phi,\zeta\ra\right|}{\left\Vert\zeta\right\Vert_{J}}=\sup_{\zeta\in W^{J}}\frac{\left|\la\psi,\genL^{x}_4 \zeta\ra_{J}\right|}{\left\Vert\zeta\right\Vert_{J}}
\end{equation*}
By the definition of the operator $\genL^{x}_4$ we have
\begin{align}
& \sum_{k=0}^{J}\int_{\R_{+}}w^{2}(\lambda)\rho^{2k}(\lambda)\psi^{(k)}(\lambda)\left(\genL^{x}_4 \zeta\right)^{(k)}d\lambda=\nonumber\\
&\hspace{1cm}=\beta^{S}\sigma_{0}(x)\sum_{k=0}^{J}\int_{\R_{+}}w^{2}(\lambda)\rho^{2k}(\lambda)\psi^{(k)}(\lambda)\left(\lambda \zeta^{(1)}(\lambda)\right)^{(k)}d\lambda\nonumber\\
&\hspace{1cm}=\beta^{S}\sigma_{0}(x)\left[\sum_{k=0}^{J}k\int_{\R_{+}}w^{2}(\lambda)\rho^{2k}(\lambda)\psi^{(k)}(\lambda)\zeta^{(k)}(\lambda)d\lambda+
\sum_{k=0}^{J}\int_{\R_{+}}w^{2}(\lambda)\rho^{2k}(\lambda)\lambda\psi^{(k)}(\lambda)\zeta^{(k+1)}(\lambda)d\lambda\right]\nonumber\\
&\hspace{1cm}\leq \left|\beta^{S}\sigma_{0}(x)\right|\left[
\sum_{k=0}^{J}\int_{\R_{+}}w^{2}(\lambda)\rho^{2k}(\lambda)\left|\lambda \psi^{(k+1)}(\lambda)+C\psi^{(k)}(\lambda)\right|\left|\zeta^{(k)}(\lambda)\right|
d\lambda\right]\nonumber\\
&\hspace{1cm}\leq \left|\beta^{S}\sigma_{0}(x)\right|
\int_{\R_{+}}\sqrt{\sum_{k=0}^{J}w^{2}(\lambda)\rho^{2k}(\lambda)\left|\lambda \psi^{(k+1)}(\lambda)+C\psi^{(k)}(\lambda)\right|^{2}}
\sqrt{\sum_{k=0}^{J}w^{2}(\lambda)\rho^{2k}(\lambda)\left|\zeta^{(k)}(\lambda)\right|^{2}}
d\lambda\nonumber\\
&\hspace{1cm}\leq \left|\beta^{S}\sigma_{0}(x)\right|
\sqrt{\sum_{k=0}^{J}\int_{\R_{+}}w^{2}(\lambda)\rho^{2k}(\lambda)\left|\lambda \psi^{(k+1)}(\lambda)+C\psi^{(k)}(\lambda)\right|^{2}d\lambda}\sqrt{\sum_{k=0}^{J}\int_{\R_{+}}w^{2}(\lambda)\rho^{2k}(\lambda)\left|\zeta^{(k)}(\lambda)\right|^{2}}
d\lambda\nonumber\\
&\hspace{1cm}= \left|\beta^{S}\sigma_{0}(x)\right|
\sqrt{\sum_{k=0}^{J}\int_{\R_{+}}w^{2}(\lambda)\rho^{2k}(\lambda)\left[\left|\lambda \psi^{(k+1)}(\lambda)\right|^{2}+C^{2}\left|\psi^{(k)}(\lambda)\right|^{2}+2C\lambda \psi^{(k+1)}(\lambda) \psi^{(k)}(\lambda)\right]d\lambda}\left\Vert\zeta\right\Vert_{J}\nonumber
\end{align}
The first inequality follows from integration by parts on the second integral and the second and third inequality by Cauchy-Schwartz using Condition \ref{A:AssumptionsOnWeights}. Therefore, by canceling the term $\left\Vert\zeta\right\Vert_{J}$ and then taking the square of the resulting expression, the statement of the lemma follows.
\end{proof}

Collecting the statements of Lemma \ref{L:Bounds2a} and Lemma \ref{L:Bounds3a} we obtain the following lemma.
\begin{lemma}\label{L:Gronwall}
For $\phi\in W^{-J}_{0}(w,\rho)$  we have
\begin{equation*}
2\la \phi, \mathcal{G}^{*}\phi\ra_{-J}+\left\Vert \genL^{*,x}_4 \phi\right\Vert^{2}_{-J} \leq  C\left(1+\left|b_{0}(x)\right|+ \left|\sigma_{0}(x)\right|^{2}+\left<\QQ,\mu\right>\right) \left\Vert \phi\right\Vert^{2}_{-J}.
\end{equation*}
\end{lemma}

Then we are in position to prove uniqueness for the limiting stochastic evolution equation.
\begin{theorem}\label{T:SEE_Uniqueness}
Let us assume that $\int_{\R_{+}} \left(w^{2}(\lambda)\rho^{2}(\lambda)\right)^{-1}\mu(d\lambda)<\infty$. Then, the solution to the stochastic evolution equation (\ref{Eq:CLT}) is unique in $W^{-J}_{0}(w,\rho)$.
\end{theorem}
\begin{proof}
Equation (\ref{Eq:UniquenessCLT3}) gives, via Lemma \ref{L:Gronwall},
\begin{eqnarray}
\BE \left\Vert \Phi_t\right\Vert^{2}_{-J} &=&  \BE\int_{0}^{t}\left(2\la\Phi_s,\mathcal{G}^{*}_{X_{s},\bar{\mu}_{s}}\Phi_s\ra_{-J}+\left\Vert \genL^{*,X_{s}}_4 \Phi_s\right\Vert^{2}_{-J}\right) ds\nonumber\\
&\leq&C \int_{0}^{t}\BE \left\Vert \Phi_s\right\Vert^{2}_{-J} ds.\nonumber
\end{eqnarray}

Therefore, by applying Gronwall inequality we get that $\BE \left\Vert \Phi_t\right\Vert^{2}_{-J}=0$, which concludes the proof of the theorem.
\end{proof}

\appendix
\section{Performance of Numerical Schemes}\label{app}
This appendix provides additional numerical results on the performance of the numerical schemes described in Sections \ref{DirectSim} and \ref{SemiScheme}.

In Figure \ref{fig5}, we compare the standard error of Schemes 1 and 2 when estimating the expectation of $f(L_T^N) = (L_T^N - S)^{+}$ (i.e., a call option on the portfolio loss). Parameters are a time step of $0.005$, strike $S = 0.12$, and $K = 6$ as the truncation level.   Scheme 1 is performed using an optimal choice of $J$ and $M$ given by (\ref{OptimalChoice}).  For small $\beta^S$, Scheme 2 outperforms Scheme 1 since the conditionally Gaussian noise dominates.  For large $\beta^S$, the process $X$ dominates the dynamics and the semi-analytic calculation reduces the standard deviation an insufficient amount to justify the additional computational time.

\begin{figure}[t]
\begin{center}
\includegraphics[scale=0.7]{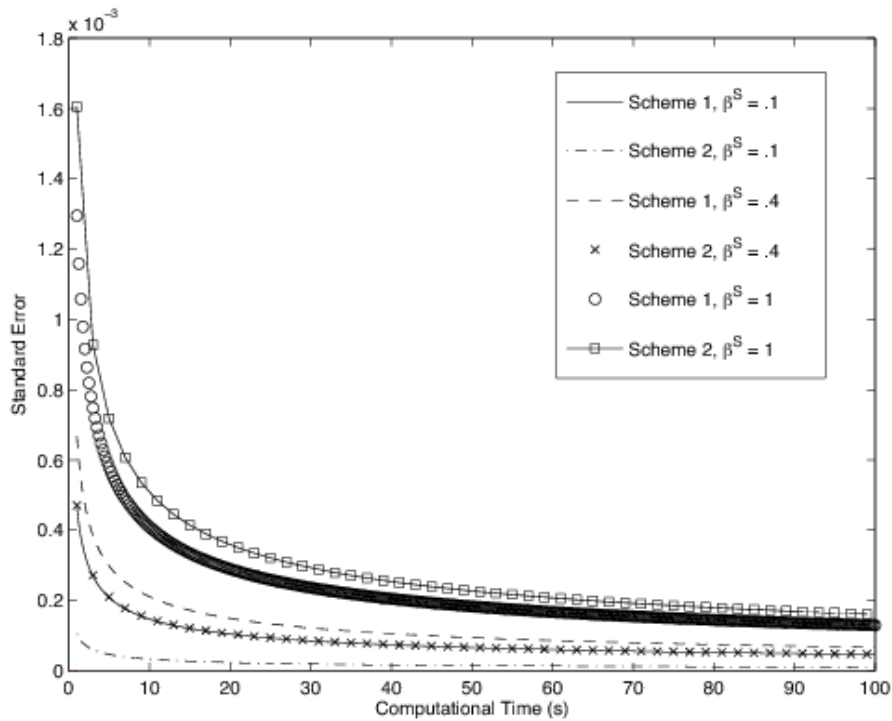}
\caption{Comparison of standard error for Scheme $1$ and Scheme $2$, respectively described in Sections \ref{DirectSim} and \ref{SemiScheme}.  Parameter case is $T = 0.5, N = 500, \sigma = 0.9, \alpha = 4, \lambda_0 = \bar{\lambda} = 0.2,$ and $ \beta^C = 1$.  The systematic risk $X$ is an OU process with mean $1$, reversion speed $2$, volatility $1$, and initial value $1$.}
\label{fig5}
\end{center}
\end{figure}

We also compare the standard error of the second-order approximation (calculated using Scheme 1) and direct simulation of the original finite system (\ref{E:main}) in Figure \ref{StandardError1}.
\begin{figure}[t]
\begin{center}
\includegraphics[scale=0.7]{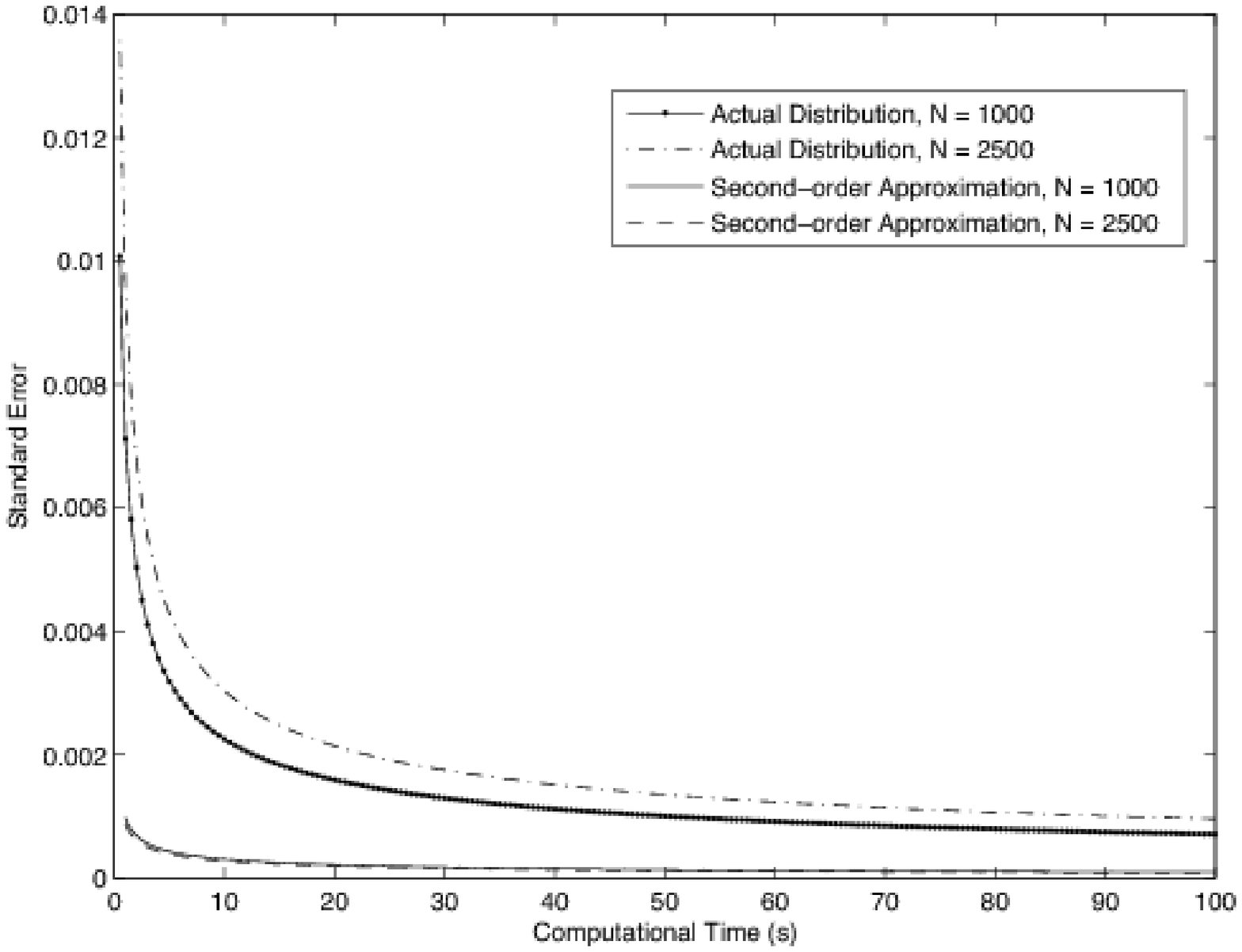}
\caption{Comparison of standard error for direct simulation of the finite system and the second-order approximation.  Parameter case is $T = 0.5, \sigma = 0.9, \alpha = 4, \lambda_0 = \bar{\lambda} = 0.2, \beta^C = 1,$ and $\beta^S =1$. The systematic risk $X$ is an OU process with mean $1$, reversion speed $2$, volatility $1$, and and initial value $1$.}
\label{StandardError1}
\end{center}
\end{figure}
We use a truncation level of $K = 6$ and a time step of $0.005$ for both asymptotic and finite systems.  We report exact values in the following table for a total computational time of $50$ seconds.  For $N = 1,000$ and $N = 2,500$, the second-order approximation has a
standard error roughly one order of magnitude smaller than the finite system simulation.  This means that for the standard errors to be equal, one must invest roughly two orders of magnitude more computational resources into the finite system simulation.  \\

\begin{center}
  \begin{tabular}{| l | c | r | }
    \hline
      & Finite System & Second-order Approximation  \\ \hline
Standard Error  ($N = 1,000$)  &  1.00 $\times 10^{-3}$ &  1.39 $\times 10^{-4}$ \\ \hline
Standard Error ($N= 2,500$) & 1.40 $\times 10^{-3}$   &  1.24 $\times 10^{-4}$  \\ \hline
  \end{tabular}
\label{Tabel1}
\end{center}


\bibliographystyle{jmr}

\end{document}